\documentclass[preprint]{imsart}

\RequirePackage[OT1]{fontenc}
\RequirePackage{amsthm,amsmath}
\RequirePackage[numbers]{natbib}
\RequirePackage[colorlinks,citecolor=blue,urlcolor=blue]{hyperref}
\usepackage{graphicx}

\startlocaldefs
\numberwithin{equation}{section}
\theoremstyle{plain}

\newtheorem{theorem}{Theorem}[section]

\newtheorem{lemma}{Lemma}[section]

\newtheorem{remark}{Remark}[section]

\endlocaldefs

\usepackage[utf8]{inputenc} 
\usepackage[T1]{fontenc}    
\usepackage{url}            
\usepackage{booktabs}       
\usepackage{amsfonts}       
\usepackage{nicefrac}       
\usepackage{microtype}  
\usepackage{mathtools}

\usepackage{amsmath,amsthm,verbatim,amssymb,amsfonts,amscd, graphicx}
\usepackage{amssymb,bm}
\usepackage{mathrsfs}
\usepackage{subfig}
\usepackage{xcolor}
\usepackage{mathtools}

\newcommand{\be}{\begin{equation*}}
\newcommand{\ee}{\end{equation*}}
\newcommand{\ben}{\begin{equation}}
\newcommand{\een}{\end{equation}}
\newcommand{\argmin}{\mathrm{argmin}}

\def\ml#1{\begin{multline*}{#1}\end{multline*}}
\def\mln#1{\begin{multline}{#1}\end{multline}}
\def\l{\left}
\def\r{\right}

\newcommand{\m}{\mathcal}
\newcommand{\mb}{\mathbb}

\newcommand{\card}{\mathrm{Card}}

\def\r{\right}
\def\l{\left}
\newcommand{\eps}{\varepsilon}
\newcommand{\var}{\mbox{Var}}

\newcommand{\wh}{\widehat}
\newcommand{\wt}{\widetilde}
\newcommand{\Ep}{\mathcal{E}_p}

\newcommand{\pr}[1]{\mathrm{Pr}{\left(#1 \right)}}

\newcommand{\scolor}[1]{\textcolor{cyan}{#1}} 

\newcommand{\rom}[1]{\uppercase\expandafter{\romannumeral #1\relax}}
\newcommand{\beas}{\begin{eqnarray*}}
\newcommand{\enas}{\end{eqnarray*}}
\newcommand{\bea}{\begin{eqnarray}}
\newcommand{\ena}{\end{eqnarray}}

\newcommand{\bml}{\begin{multline*}}
\newcommand{\eml}{\end{multline*}}
\newcommand{\bmln}{\begin{multline}}
\newcommand{\emln}{\end{multline}}

\newcommand{\bels}{\begin{align*}}
\newcommand{\enls}{\end{align*}}
\newcommand{\bel}{\begin{align}}
\newcommand{\enl}{\end{align}}

\newcommand{\ignore}[1]{}

\newcommand{\med}{\mathrm{median}}
\newcommand{\I}{\mathbf{1}}

\begin{document}

\begin{frontmatter}
\title{Robust and efficient mean estimation: an approach based on the properties of self-normalized sums}
\runtitle{Robust mean estimation}


\begin{aug}
\author{\fnms{Stanislav} \snm{Minsker}\ead[label=e1]{minsker@usc.edu}}

\address{Department of Mathematics,\\ University of Southern California\\
\printead{e1}}

\author{\fnms{Mohamed} \snm{Ndaoud}\ead[label=e2]{ndaoud@essec.edu}}

\address{Department of Decisions Sciences,\\
ESSEC Business School\\
\printead{e2}}

\runauthor{S. Minsker and M. Ndaoud}

\end{aug}

\begin{abstract}
Let $X$ be a random variable with unknown mean and finite variance. 
We present a new estimator of the mean of $X$ that is robust with respect to the possible presence of outliers in the sample, provides tight sub-Gaussian deviation guarantees without any additional assumptions on the shape or tails of the distribution, and moreover is asymptotically efficient. This is the first estimator that provably combines all these qualities in one package. Our construction is inspired by robustness properties possessed by the self-normalized sums. 
Theoretical findings are supplemented by numerical simulations highlighting strong performance of the proposed estimator in comparison with previously known techniques.
\end{abstract}


\begin{keyword}
\kwd{robust estimation}
\kwd{sub-Gaussian deviations}
\kwd{self-normalized sums}
\kwd{efficiency}
\end{keyword}
\tableofcontents
\end{frontmatter}

\mathtoolsset{showonlyrefs=true}
\section{Introduction.}
\label{sec:intro}


Let $X$ be a random variable with mean $\mb EX=\mu$ and variance $\var(X)=\sigma^2$, where both $\mu$ and $\sigma^2$ are unknown; in what follows, $P$ will denote the distribution of $X$ and $\m P_{2,\sigma}$ -- the class of all distributions possessing $2$ finite moments and having variance $\sigma^2$. 
We will be interested in robust estimators $\wh \mu$ of $\mu$ constructed from the data $X_1,\ldots,X_N$ generated as follows: the initial non-corrupted sample $X_1,\ldots,X_{N'}$ of independent, identically distributed copies of $X$ is merged with a set of $ O<N'$ outliers that are independent from the initial sample, and the combined sample of cardinality $N :=N'+ O$ is given as an input to an algorithm responsible for construction of the estimator. 
This contamination framework is more general than Huber's contamination model \cite{huber1964robust,chen2016general} where the outliers are assumed to be identically distributed, but weaker than the framework allowing adversarial outliers \cite{kearns1993learning,valiant1985learning} that may for instance depend on the initial sample. 
Robustness will be quantified by two properties: first, in the situation when $ O=0$, the estimators should admit tight non-asymptotic deviation bounds of the form 
\ben
\label{eq:deviations}
\l| \wh\mu - \mu\r|\leq C\sigma \sqrt{\frac{s}{N}}
\een
with probability at least $1-2e^{-s}$, where $C>0$ is an absolute constant. 
In particular, we will be interested in the estimators that attain such deviation guarantees uniformly over $0<s<\psi_P(N)$ where 
$\psi_P(N)$ is an increasing function that might depend on the law of $X$ \footnote{It follows from results in \cite{devroye2016sub} that the function $\psi_P(N)$ can not be chosen to be independent of $P$, no matter how slow its growth is. At the same time, our results show that for every $\sigma>0$ and $P\in \m P_{2,\sigma}$, such a function exists.}; guarantees of type \eqref{eq:deviations} can be informally labeled as ``robustness to heavy tails.'' 
Second, the estimators of interest should perform optimally with respect to the degree of outlier contamination characterized by the quantity $\eps:=\frac{ O}{N}$.  

Another important property that we focus on is \emph{asymptotic efficiency}. 
Informally speaking, efficiency measures how ``wasteful'' an estimator is: an efficient estimator will capture all the information available in the sample; alternatively, in many cases it is possible to conclude that the confidence intervals centered at an efficient estimator will have (at least asymptotically) smallest possible diameter. 
It is difficult to quantify efficiency using only finite-sample guarantees of type \eqref{eq:deviations} as the constants in these bounds are rarely sharp, at least, for practical considerations, and therefore a common approach is to take an asymptotic viewpoint. 
Specifically, we will be looking for the estimators that are asymptotically normal and have asymptotic variance that is as small as possible in the minimax sense, that is, 
$\sqrt{N}\l( \wh\mu - \mu\r)\xrightarrow{d} \m N(0,\nu^2)$ as $N\to\infty$, where $\xrightarrow{d}$ denotes convergence in distribution and $\nu^2:=\nu^2(\wh\mu,P)$ is such that 
\be
\sup_{P\in\m P_{2,\sigma}} \nu^2(\wh\mu,P) = \inf_{\wt \mu} \sup_{P\in\m P_{2,\sigma}} \nu^2(\wt\mu,P).
\ee
Here, the infimum is taken over all asymptotically normal (after rescaling by $\sqrt N$) estimators $\wt\mu$ of $\mu$. 
It is easy to see that $\inf_{\wt \mu} \sup_{P\in\m P_{2,\sigma}} \nu^2(\wt\mu,P) = \sigma^2$ (for reader's convenience, the proof of this  simple fact is given in Lemma \ref{lemma:minmax}), therefore, it suffices to find a robust estimator that satisfies 
$\sqrt{N}\l( \wh\mu - \mu\r)\xrightarrow{d} \m N(0,\sigma^2)$ for all $P\in \m P_{2,\sigma}$. 
For instance, the sample mean is an example of the estimator with required asymptotic properties that is not robust, while the popular median-of-means estimator \cite{Nemirovski1983Problem-complex00} is robust but not asymptotically efficient \cite{minsker2019distributed}. 

In this paper we construct the first, to the best of our knowledge, example of an estimator of the mean that is provably \textbf{(a)} robust to the heavy tails of the data-generating distribution $P$; \textbf{(b)} admits optimal error bounds with respect to the outlier contamination proportion $\eps = \frac{O}{N}$; \textbf{(c)} is asymptotically efficient and \textbf{(d)} is almost tuning-free, meaning that it does not require information about any parameters of the distribution besides the upper bound for the contamination proportion $\eps$. We also show how to make our procedure fully adaptive. Our construction is novel and is inspired by the properties of self-normalized sums. 

The rest of the paper is organized as follows: section \ref{sec:construction} introduces the estimator and explains the main ideas behind its construction; the key results are presented in section \ref{sec:main}, while comparison of our estimator with existing robust estimation techniques in the context of properties \textbf{(a)} - \textbf{(d)} is presented in section \ref{sec:comparison}. Finally, a fully adaptive procedure is outlined in section \ref{sec:adaptation} while the supporting numerical simulations are included in section \ref{sec:simul}. The proofs of the main results are contained in section \ref{sec:proofs}. 
All notation and auxiliary results will be introduced on demand.


\section{Construction of the estimator.}
\label{sec:construction}

We restrict our attention to the estimators that are obtained via aggregating the sample means evaluated over disjoint subsets (also referred to as ``blocks'') of the data. Specifically, assume that $\{1,\ldots,N\} = \bigcup_{j=1}^k G_j$ where $G_i\cap G_j=\emptyset$ for $i\ne j$ and $|G_j| = n= N/k$ is an integer, and let $\bar \mu_j:=\frac{1}{|G_j|}\sum_{i\in G_j} X_i$ be the sample mean of the observations indexed by $G_j$. We consider estimators $\wh\mu_N$ of the form
\ben
\label{eq:estimator}
\wh\mu_N = \sum_{j=1}^k \alpha_j \bar\mu_j
\een
for some (possibly random and data-dependent) nonnegative weights $\alpha_1,\ldots,\alpha_k$ such that $\sum_{j=1}^k \alpha_j=1$. 
For example, the well known median-of-means estimator \cite{Nemirovski1983Problem-complex00,alon1996space,lerasle2011robust} corresponds to the case $\alpha_{\wh j} = 1$ for $\wh j$ such that $\bar\mu_{\wh j} = \med\l(\bar\mu_1,\ldots,\bar \mu_k\r)$ and $\alpha_{\wh j} = 0$ otherwise. 
Construction proposed in this paper starts with an observation that  it is natural to choose the weights that are inversely proportional to some increasing function of the standard deviation of each block. Indeed, the estimation error of the sample mean $\bar \mu_j$ in each block of the data is essentially controlled by the corresponding sample standard deviation $\wh \sigma_j:=\sqrt{\frac{1}{|G_j|}\sum_{i\in G_j} (X_i - \bar \mu_j)^2}$. 
To understand why, consider the following obvious identity:
\[
\l|\bar \mu_j - \mu\r| = \l| \frac{\bar \mu_j - \mu}{\wh\sigma_j}\r|\wh\sigma_j.
\]
The random variable $\frac{\bar \mu_j - \mu}{\wh{\sigma}_j}$, which is equal up to normalization to the Student's t-statistic, is known to be tightly concentrated around $0$: namely, it is bounded by $\l|\sqrt{\frac{t}{n}}\r|$ with probability at least $1-e^{-ct}$ for $t\leq c'n$ where $c,c'$ are positive constants, even if data are heavy-tailed (a more precise version of this fact is stated below). 
Therefore, $\l|\bar \mu_j - \mu\r|$ is bounded by a multiple of $\frac{\wh\sigma_j}{\sqrt n}$ with high probability. 
And, while the error $\l|\bar \mu_j - \mu\r|$ is unknown, the quantity $\wh\sigma_j$ is fully data-dependent. 
This motivates the choice of the weights of the form 
\ben
\label{eq:weights}
\alpha_j = \frac{1/\wh\sigma^p_j }{\sum_{i=1}^k 1/\wh\sigma^p_i}
\een
for some $p\geq 1$; in what follows, the estimator \eqref{eq:estimator} with weights \eqref{eq:weights} will be denoted 
$\wh{\mu}_{N,p}$. When we need to emphasize the specific value of $k$ used in the construction, we will write $\wh{\mu}_{N,p}(k)$. 
Observe that when $p=1$, the estimation error satisfies 
\ben
\label{eq:err}
\wh\mu_{N,1} - \mu = \frac{\frac{1}{k}\sum_{j=1}^k \frac{\bar\mu_j-\mu}{\wh\sigma_j}}{\frac{1}{k}\sum_{j=1}^k \frac{1}{\wh\sigma_j}},
\een
which is proportional to the average of t-statistics evaluated over $k$ independent subsamples. It is therefore natural to expect that $\wh\mu_{N,1} - \mu$ will satisfy strong deviation guarantees. 

Let us present now an example where the weights corresponding to $p=2$ arise naturally. Observe that one can model outliers by assuming that the variances of the data differ across $k$ groups, where large variance corresponds to a corrupted subsample: $X_i, \ i\in G_j\sim N(\mu,\sigma_j^2)$ for some $\mu\in \mb R$ and positive but unknown $\sigma_1,\ldots,\sigma_k$. The maximum likelihood estimator $\wt\mu$ in this model is easily seen to satisfy
$\wt\mu = \argmin_{z\in \mb R}\sum_{j=1}^k |G_j| \log\l( \sum_{i\in G_j}\l( X_i - z\r)^2\r)$. 
Equivalently, $\wt \mu$ can be defined via 
\[
\wt \mu = \argmin_{z\in \mb R}\sum_{j=1}^k |G_j| \log\l(1 + \l( \frac{\bar \mu_j - z}{\wh\sigma_j}\r)^2 \r).
\]
An approximate solution can be obtained via minimizing the first-order approximation of the loss function $z\mapsto \sum_{j=1}^k |G_j|\l( \frac{\bar \mu_j - z}{\wh\sigma_j}\r)^2$ that attains its minimum at the point
\[
\sum_{j=1}^k \bar \mu_j \frac{|G_j|/\wh\sigma_j^2}
{\sum_{i=1}^k |G_i|/\wh\sigma_i^2} =
\sum_{j=1}^k \bar \mu_j \frac{1/\wh\sigma_j^2}{\sum_{i=1}^k 1/\wh\sigma_i^2},
\]
which is the estimator \eqref{eq:estimator} with weights defined in \eqref{eq:weights} for $p=2$. 
In the following sections we will present non-asymptotic deviation bounds for the estimator $\wh\mu_{N,p}$ for all values of $p\geq 1$ and will establish its asymptotic efficiency in the absence of outliers. 

\section{Main results.}
\label{sec:main}

The goal of this section is to prove the deviation inequality for the estimation error $\wh\mu_{N,p} - \mu$ for any $p\geq1$, where the estimator $\wh\mu_{N,p}$ corresponds to the weights defined by \eqref{eq:weights}.

\subsection{Preliminaries.}

In this section, we consider the simple framework of i.i.d data without outliers. We will start with a brief review of concentration inequalities for the self-normalized sums. 
It is known (for example, see the book \cite{pena2008self}) that the properties of the t-statistics 
\ben
\label{eq:t-stat}
T_j:= \begin{cases} \frac{\bar{\mu}_j - \mu}{\wh{\sigma}_j}, & \wh\sigma_j > 0, 
\\
0, & \wh\sigma_j=0
\end{cases}
\een
evaluated over subsamples indexed by $G_1,\ldots,G_k$ are closely related to the behavior of the self-normalized sums 
$Q_j:=\frac{\bar \mu_j - \mu}{V_j}$ where $V_j^2 := \frac{1}{|G_j|}\sum_{i\in G_j} (X_i - \mu)^2$. 
Indeed, it is easy to see that $T_j = f(Q_j)$ where $f(z) = \frac{z}{\sqrt{1-z^2}}$. 
The following inequality is well known (cf. Theorem 2.16 \cite{pena2008self}): for any $j=1,\ldots,k$ and any $x>0$, \footnote{Since $|Q_j|\leq 1$, the inequality is nontrivial only for $x < \sqrt{|G_j|}$. }
\ben
\label{eq:sns-2}
\l|Q_j \r| \leq \frac{x}{\sqrt{|G_j|}}\l( 1 + \frac{4\sigma}{V_j}\r) 
\een
with probability at least $1-4e^{-x^2/2}$, as long as $\mb E(X-\mu)^2 < \infty$. 
In order to deduce a non-random upper bound from \eqref{eq:sns-2}, it suffices to control the ratio $\frac{1}{V_j}$. 
To this end, define
\be
\zeta(X) := \inf \l\{ a>0 : \mb E \l(|X-\mu|^{2} \I\{|X-\mu|\leq \sigma \cdot a\}\r) \geq \sigma^2/2 \r\}.
\ee
As long as $\var(X)$ is finite, it is clear that $\zeta(X)<\infty$. 
\begin{lemma}\label{lem:sns-2}
With probability at least $1-e^{-\frac{n}{40\zeta^2(X) \vee 6}}$, $V_j \geq \frac{\sigma}{2}$. 
\end{lemma}
\noindent Combining this inequality with the bound \eqref{eq:sns-2}, we deduce that for any $1\leq j\leq k$ and any $x>0$, 
\ben\label{eq:sns-6}
|Q_j| \leq \frac{9x}{\sqrt{n}},
\een
with probability at least $1 - 4e^{-x^2/2} - e^{-\frac{n}{40\zeta^2(X) \vee 6}}$. 
If moreover $x\leq \sqrt{n}/ 18 $, then the relation $T_j = f(Q_j)$ immediately implies that t-statistics $T_j$ satisfy the bound
\ben
\label{eq:sns-5}
|T_j| \leq \frac{11x}{\sqrt{n}}
\een
with probability at least $1 - 4e^{-x^2/2} - e^{-\frac{n}{40\zeta^2(X) \vee 6}}$ for each $1\leq j\leq k$. 
Alternatively, the previous argument also implies that the random variables 
$T_j \I\{ |Q_j| \leq 1/2, \ V_j\geq \sigma/2 \}$ satisfy the deviation inequality 
\be
\pr{\l| T_j \I\{ |Q_j| \leq 1/2, \ V_j\geq \sigma/2 \}\r| \geq \frac{11x}{\sqrt n}}\leq 4e^{-x^2/2}. 
\ee
Therefore, we conclude that the random variable $T_j$, truncated at the right level, behaves like a sub-Gaussian random variable.\footnote{$X$ has sub-Gaussian distribution if $\exists K>0$ such that $\forall p\geq 1\;  \mb E(|X-\mu|^{p})^{1/p} \leq K\sqrt{p})$. See section 2.5 in \cite{vershynin2018high} for the details.} This fact is formalized in Lemma \ref{lem:concentration2} and is one of the key ingredients used to show that proposed estimators have sub-Gaussian deviations.

\subsection{Non-asymptotic deviation inequalities.}

In the simplest case $p=1$, equation \eqref{eq:err} suggests that in order to bound the estimation error $\wh\mu_{N,p} - \mu$, it suffices to control the average
$\frac{1}{k}\sum_{j=1}^k \frac{\bar\mu_j-\mu}{\wh\sigma_j}$ and the harmonic mean 
$\l( {\frac{1}{k}\sum_{j=1}^k \frac{1}{\wh\sigma_j}} \r)^{-1}$ separately. 
Similar intuition holds for other values of $p$ as well. In what follows, we will always assume that $ O \leq Ck$ for some $C<1$, where $ O$ is the number of outliers in the sample. Define the event
\ben
\label{eq:Ep}
\Ep := \l\{ \l(\frac{1}{k}\sum_{j=1}^k \frac{1}{\wh \sigma^p_j}\r)^{-1}  \leq \l(\frac{4\sigma}{1-C}\r)^{p} \r\}.
\een
 $\Ep$ holds whenever the harmonic mean of the (powers of) sample variances does not exceed the corresponding power of the true variance $\sigma^2$ by too much. In particular, in the absence of outliers, we can replace $C$ by $0$ in the previous event. Informally speaking, the harmonic mean of a set of numbers is controlled by its smallest elements, therefore, it is natural to expect that the event $\Ep$ holds with overwhelming probability; this claim will be formalized in the following lemma whose proof is deferred to Section~\ref{proof:lem:denominator}. 
\begin{lemma}
\label{lem:denominator}
Recall the contamination framework defined in Section \ref{sec:intro}. Suppose that $\mb E|X-\mu|^{1+\delta}<\infty$ for some $1\leq \delta \leq 2$ and that $ O \leq Ck$ for some $C<1$. Then
\be
\pr{\Ep}
\geq 1 - e^{-ck(1-C)(1 + (\delta-1)\log{n})}
\ee
for some constant $c>0$ that depends on $\delta$ and $\mb E|X-\mu|^{1+\delta}$. Moreover, if $X$ has sub-Gaussian distribution, then
\be
\pr{\Ep}
\geq 1 - e^{-c(1-C)N}
\ee
for some constant $c>0$ that depends on the distribution of $X$.
\end{lemma}
Note that the condition $ O < Ck$ only requires $C$ to be smaller than $1$: it means that for our technique to reliably estimate the true mean, it suffices that any constant positive fraction of subsamples indexed by $G_1,\ldots,G_k$ are free from the outliers, while the popular median-of-means estimator requires at least $50\%$ of the subsamples to be ``clean''. In practical applications, this difference can be substantial, and our simulation results (see section \ref{sec:simul}) confirm this observation.

Our first result presents non-asymptotic deviation bounds for the case when the sample does not contain outliers. 
\begin{theorem}
\label{thm:deviation}
Suppose that $\mb E|X-\mu|^{1+\delta}<\infty$ for some $1\leq \delta \leq 2$. Then with probability at least $1-2e^{-s} - ke^{-cn} - \Pr(\Ep^{c})$,
\ben
\label{eq:sub-Gauss}
|\wh{\mu}_{N,p} - \mu| \leq C_p\sigma \l( \sqrt{\frac{s+1}{N}} + \phi(\delta,n) \r) 
\een
where $c>0$ depends only on $\zeta(X)$, $C_p>0$ depends only on $p$, and 
\be
\phi(\delta,n) = \begin{cases}
o(n^{-\delta/2}), &\delta<2, \\ O(n^{-1}), & \delta = 2
\end{cases}
\ee
as $n\to\infty$.
\end{theorem}
\noindent Combination of Theorem \ref{thm:deviation} with Lemma \ref{lem:denominator} readily implies that $\wh\mu_{N,p}$ admits sub-Gaussian deviation guarantees for $s = k \lesssim \sqrt{N/\log{N}}$. Indeed, in that case we get with probability at least $1-3e^{-k}$ that 
\be
|\wh{\mu}_{N,p} - \mu| \leq C'_p\sigma  \sqrt{\frac{k+1}{N}} .
\ee
As we explain in the remark below, if $k$ is chosen appropriately, this statement can often be strengthened to yield uniform deviation guarantees holding in the range $0 \leq s \leq k$. 
\begin{remark}
Dependence of the constant $c$ on $\zeta(X)$ is inherited from Lemma \ref{lem:sns-2}. The constant $\zeta(X)$ can be arbitrary large, therefore the inequality of Theorem \ref{thm:deviation} does not hold with overwhelming probability uniformly over the class of distributions $\m P_{2,\sigma}$. To achieve uniformity, we need to assume slightly more about the distribution of $X$ -- for example, one may impose the ``small ball'' condition $Q(u):=\pr{|X-\mb EX|\geq u}\geq \tilde c>0$, or the equivalence of  moments of order $2$ and $2+\delta$ for some $\delta>0$, namely that $\mathbb E|X-\mathbb EX|^{2+\delta}\leq \tilde C (\mathbb E|X-\mathbb EX|^2)^{1+\delta/2}$ for some fixed $\tilde C>0$. Then our bounds will depend on the constant $\tilde c$ or $\tilde C$ instead, and dependence on $\zeta(X)$ can be suppressed: for instance, when the moments of order $2$ and $2+\delta$ are equivalent, we have that $\mathbf{E}(|X-\mu|^2\mathbf{1}\{|X-\mu|\geq \sigma(2 C)^{1/\delta} \}) \leq \sigma^2/2$ in view of Markov's inequality, thus $\zeta(X) \leq (2C)^{1/\delta}$. This justifies the claim that assuming $\zeta(X)$ to be ``small'' is a relatively mild requirement. In simple terms, we ask that the distributions in question assign non-trivial mass to a fixed neighborhood of their means. 
It is also interesting to take a viewpoint that assumes the distribution of $X$ to be fixed while the parameters $n,k\to\infty$: in this case, one can establish stronger claims about mean estimation -- for instance, the deviations in \eqref{eq:deviations} can be shown to be uniform over a range of values of parameter $s$. 
\end{remark}
\begin{remark}
\label{remark:1}
A more precise bound for the ``bias term'' $\phi(\delta,n)$ has the form 
\[
\phi(\delta,n) = n^{-\delta/2}\cdot\l( n^{-\frac{2-\delta}{4}} \vee g^{\frac{2-\delta}{2+\delta}}(n^{1/4})\r),
\]
where $g(u) = \mb E \l(|X-\mu|^{1+\delta}\mathbf{1}_{\{|X-\mu|\geq u\}} \r)$. 
It is therefore easy to see that whenever $k = o\l( N^{\frac{\delta-1}{\delta}}\r)$, the term $\phi(\delta,n)$ is $o(N^{-1/2})$ 
and the sub-Gaussian deviation guarantees \eqref{eq:sub-Gauss} hold uniformly over $s\lesssim k$ (the latter restriction appears due to the fact that the probability of event $\Ep$ depends on $k$ as $e^{-ck}$). 

In the case when $\delta=1$, $\phi(\delta,n) = o\l( \sqrt{\frac{k}{N}}\r)$ so that sub-Gaussian deviation guarantees hold with $s=k\lesssim n$. However, if $k$ is large enough, namely, if $k\l( n^{-\frac{1}{4}} \vee g^{\frac{1}{3}}(n^{1/4})\r)^2 = O(1)$, we can still achieve the situation when $\phi(\delta,n)=O\l(N^{-1/2}\r)$. In this case,  deviation guarantees hold uniformly over $s\lesssim k$. 
The price that we have pay however is the fact that $k$ can grow arbitrarily slowly as a function of $N$, but this is unavoidable in general as shown in \cite{devroye2016sub}.  
\end{remark}
 Next, we discuss the more general contamination framework described in the introduction. For each block $G_j$, we denote by $W_j$ the number of outliers in $G_j$ and by $\bar{\mu}_{j}^{I}$ (respectively $\bar{\mu}_{j}^{C}$) the sample mean corresponding to the inliers (respectively outliers) within $G_j$. For every set of outliers $\mathcal{O}$ we define
 \ben\label{eq:outlier_alpha}
 \alpha(\mathcal{O}):= 1 + \underset{j : W_j \neq 0}{\min}\frac{W_j(\bar{\mu}_{j}^I - \bar{\mu}_{j}^C)^2}{n\sigma^2}.
 \een
Informally speaking, $\alpha(\mathcal{O})$ can be viewed as a proxy for the magnitude of the outliers. The following extension of Theorem \ref{thm:deviation} holds.
\begin{theorem}
\label{thm:outlier}
Suppose that $\mb E|X-\mu|^{1+\delta}<\infty$ for some $1\leq \delta \leq 2$, and $O\leq Ck$ for some $C<1$. 
Then with probability at least $1-2e^{-s} - ke^{-cn} - \Pr(\Ep^{c})$,
\be
|\wh\mu_{N,p} - \mu| \leq \frac{C_p\sigma}{(1-C)^p} \l( \sqrt{\frac{s+1}{N}} + \phi(\delta,n)+\alpha(\m O)^{-(p-1)/2} \frac{ O}{k\sqrt{n}} \r) 
\ee
for $c>0$ depending only on $\zeta(X)$, $C_p>0$ depending only on $p$, $\alpha(\mathcal{O})$ defined via \eqref{eq:outlier_alpha}, and $\phi(\delta,n)$ defined as in Theorem \ref{thm:deviation}. 
\end{theorem}
One may notice that $\alpha(\m O)^{-(p-1)/2} \leq 1$, and this quantity gets smaller as $p$ grows, suggesting that the estimator $\wh{\mu}_{N,p}$ is more robust to the outliers of large magnitude as $p$ increases. 
Next, let us discuss the term $\frac{ O}{k\sqrt n} = \eps \sqrt{n}$ that quantifies dependence of the estimation error on the fraction  of outliers $\eps = \frac{ O}{N}$. It is easy to see that the ``best'' choice of $k$ for which the terms $\phi(\delta,n)$ and $\frac{O}{k\sqrt n}$ are of the same order is $k\propto N \eps^{\frac{2}{1+\delta}}$ yielding the error rate of $\eps^{\frac{\delta}{1+\delta}}$ that is known to be optimal with respect to $\delta$ (e.g. see section 1.2 in \cite{steinhardt2017resilience} or Lemma 5.4 in \cite{minsker2018uniform}). 
However, as the upper bound depends explicitly on the magnitude of outliers through $\alpha(\m O)$, in some scenarios it can be much smaller than the worst case given by $\frac{O}{k\sqrt{n}}$. 

\subsection{Asymptotic efficiency.}

The following result establishes asymptotic efficiency (in a sense defined in section \ref{sec:intro}) of the estimator $\wh{\mu}_{N,p}$ for any $p\geq 1$ in the absence of outliers, implying that the estimator can not be uniformly improved in general.
\begin{theorem}
\label{thm:clt}
Suppose that $\mb E|X-\mu|^{1+\delta} < \infty$ for some $1\leq \delta \leq 2$. 
Let $\{k_j\}_{j\geq 1}\subset \mb N$, $\{n_j\}_{j\geq 1}\subset \mb N$ be two non-decreasing, unbounded sequences satisfying $\sqrt{N_j}\phi(\delta,n_j) = o(1)$ as $j\to\infty$, where $N_j:=k_j n_j$ and $\phi(\delta,n)$ was defined in remark \ref{remark:1}. 
Then for any $p\geq 1$,
\[
\sqrt{N_j}\l(\wh\mu_{N_j,p} - \mu \r) \xrightarrow{d} \mathcal{N}\l(0,\sigma^2\r) \text{ as } j\to\infty.
\]
\end{theorem}
Condition $\sqrt{N_j}\phi(\delta,n_j) = o(1)$ is essentially a requirement that the bias of estimator $\wh{\mu}_{N_j,p}$ is asymptotically of order $o\l( N_j^{-1/2}\r)$. 
It is not difficult to check that the sequences $\{k_j\}_{j\geq 1}, \ \{n_j\}_{j\geq 1}$ with required properties exist for any distribution $P\in \m P_{2,\sigma}$, see remark \ref{remark:1} for the details. 
For example, if $\mb E|X-\mu|^3<\infty$, it suffices to require that $k_j = o(n_j)$. 

Together, results of section \ref{sec:main} imply that the estimator $\wh\mu_{N,p}$ can be viewed as a true robust alternative to the sample mean -- it preserves its desirable properties such as asymptotic efficiency while being robust at the same time.

\subsection{Comparison with existing techniques.}
\label{sec:comparison}

One of the most well-known consistent, robust estimators of the mean in the class $\m P_{2,\sigma}$ is the median-of-means estimator \cite{Nemirovski1983Problem-complex00,alon1996space,lerasle2011robust}. 
While it is robust to heavy tails, adversarial contamination, and is tuning-free, it is not asymptotically efficient: indeed, according to Theorem 4 in \cite{minsker2019distributed}, the asymptotic variance of the median-of-means estimator is $\frac{\pi}{2}\sigma^2$. This fact is illustrated in our numerical experiments in section \ref{sec:simul}. 
Another family of estimators belonging to the broad class defined via equation \eqref{eq:estimator} is discussed in section 2.4 in \cite{minsker2019distributed} and is defined via 
\[
\wt \mu_N = \argmin_{z\in \mb R} \sum_{j=1}^k \rho\l( \frac{\sqrt{n}}{\Delta}(\bar \mu_j - z)\r)
\]
where $\rho$ is Huber's loss $\rho(z) = \min\l( \frac{z^2}{2}, |z| - \frac{1}{2}\r)$ and $\Delta>0$. The asymptotic variance of this estimator can be made arbitrarily close to $\sigma^2$, however, achieving this requires $\sigma^2$ to be known. 

Construction of Catoni's estimator \cite{catoni2012challenging} again requires knowledge of $\sigma^2$ (or its tight upper bound), moreover, it is not robust to adversarial contamination. 
Finally, deviation bounds for the trimmed mean estimator obtained in \cite{lugosi2019robust} are not uniform with respect to the confidence parameter $s$ (meaning that different choices of $s$ require the estimator to be re-computed), and its asymptotic efficiency, while plausible, has not been formally established. Moreover, construction employed in \cite{lugosi2019robust} requires sample splitting. Recently, Lee and Valiant \cite{lee2020optimal} showed that it is possible to construct a mean estimator that achieves sub-Gaussian guarantees with essentially optimal constants, however, their estimator explicitly depends on the desired confidence level, and its asymptotic behavior is not discussed.

The only other robust, tuning free estimator that is asymptotically efficient, albeit only for a \emph{subclass} of $\m P_{2,\sigma}$, is a permutation-invariant version of the median-of-means estimator (which is also the higher order Hodges-Lehmann estimator). 
It is defined as follows: let $\m A_N^{(n)}:=\l\{  J: \ J\subseteq \{1,\ldots,N\}, \card(J)=n\r\}$ be a collection of all distinct subsets of $\{1,\ldots,N\}$ of cardinality $n$,  $\bar\theta_J:=\frac{1}{n}\sum_{j\in J} X_j$, and 
\(
\wt \mu_U:=\med\l(\bar\theta_J, \ J\in \m A_N^{(n)}\r).
\) 
We note that $\card\l( \m A_N^{(n)}\r) = {N\choose n}$, so that for large $N$ and $n$ exact evaluation of $\wt \mu_U$ is not computationally feasible. 
The following result was established recently in \cite{diciccio2020clt}: assume that $N_j = n_j k_j$ is the sample size where $n_j, \ k_j\to\infty$ as $j\to\infty$ such that $n_j = o\l(\sqrt{N_j}\r)$. Moreover, suppose that $X$ is normally distributed with mean $\mu$ and variance $\sigma^2$. Then 
$\sqrt{N}\l( \wt \mu_U - \mu \r)\xrightarrow{d} \m N\l( 0,\sigma^2 \r)$. 
While is likely that the result still holds for other symmetric distributions, the condition $n_j = o\l(\sqrt{N_j}\r)$ is restrictive: for example, for non-symmetric distribution possessing 3 finite moments, the bias of the estimator $\wt \mu_U$ is of order $n_j^{-1}$, and the requirement $n_j = o\l(\sqrt{N_j}\r)$ implies that this bias is asymptotically larger than $N_j^{-1/2}$. 

Finally, there is a growing body of literature related to sub-Gaussian mean estimators in $\mb R^d$, for example see the papers \cite{dalalyan2020all,lugosi2019mean,hopkins2020mean}, and references therein. These works are mainly concerned with rate optimality, and questions related to asymptotic efficiency have not been investigated in detail. 

\section{Adaptation to the contamination proportion $\eps$.}
\label{sec:adaptation}

The number of outliers $O$ is usually unknown in practice, therefore, it is desirable to have a procedure that can adapt to this unknown quantity. Fortunately, the proposed method admits a natural adaptive version. 
This extension is based on the following observation: assume that $p=1$, and consider the estimation error
$\wh\mu_{N,1} - \mu = \frac{\frac{1}{k}\sum_{j=1}^k \frac{\bar\mu_j-\mu}{\wh\sigma_j}}{\frac{1}{k}\sum_{j=1}^k \frac{1}{\wh\sigma_j}}$. Then the numerator of this expression admits an upper bound 
$\l| \frac{1}{k}\sum_{j=1}^k \frac{\bar\mu_j-\mu}{\wh\sigma_j} \r| \leq C_p\sigma \l( \sqrt{\frac{s}{N}} + \sqrt{\frac{O}{N}}\r)$ that holds \emph{for all} choices of $k$ with probability at least $1- 2e^{-s} - k e^{-cn}$, as shown in the proof of Theorem \ref{thm:outlier}. Therefore, it suffices to choose $k$ such that the harmonic mean 
$\frac{k}{\sum_{j=1}^k \frac{1}{\wh\sigma_j}}$ is a good, in a relative sense, estimator of $\sigma$. Fortunately, the harmonic mean of standard deviations is a fully data-dependent quantity that can be evaluated for any $k$; similar intuition holds for other values of $p$ as well. 

Based on the previous observation, we propose an adaptive estimator $\wt{\mu}_p$ defined as follows. 
We will choose $k$ as the smallest integer, on a logarithmic scale, which guarantees that $\frac{k}{\sum_{j=1}^k \frac{1}{\wh\sigma^p_j}}$ is not too large compared to $\sigma^p$, in a sense defined by \eqref{eq:Ep}. 
To this end, we only need to obtain a good preliminary estimator of $\sigma$ that we can compare the harmonic means to. Assume that we are already given an estimator $\wt{\sigma}$ such that
\ben\label{eq:variance_est}
 1/20 \leq \frac{\wt{\sigma}}{\sigma} \leq 4
\een
with large probability. The above assumption is not restrictive since, as we will show in section \ref{app:robust_variance}, one can construct $\wt \sigma$ such that \eqref{eq:variance_est} holds with probability at least $1-e^{-cN}$ for some absolute $c>0$, under mild conditions. Next, for each positive integer $k$, set
\be
\wt{\Ep}(k) := \l\{\l(\frac{1}{k}\sum_{j=1}^k \frac{1}{\wh \sigma^p_j}\r)^{-1}  \leq \l(\frac{80\wt{\sigma}}{1-C}\r)^{p}\r\}.
\ee
Finally, define $\tilde{k}$ via
\(
\log_2{\tilde{k}} := \inf\l\{ i \in \{1,\dots,\lfloor\log_2{N}\rfloor\} \,:\, \wt{\Ep}(2^i) \text{ holds }  \r\} \vee 1,
\) 
\footnote{We assume that the infimum over the empty set is equal to $-\infty$.}
and the corresponding estimator $\wt{\mu}_p(s) := \wh\mu_{N,p}(\tilde{k} \vee \lfloor s \rfloor+1)$. The following bound is the main result of this section; essentially, it states that $\wt \mu_p(s)$ is a robust estimator that is fully adaptive and provides sub-Gaussian deviation guarantees.  
\begin{theorem}
\label{thm:outlier_adaptive}
Suppose that $\mb E|X-\mu|^{2}<\infty$. Assume that $2 \leq O \leq N/4$ and that $\wt{\sigma}$ satisfies \eqref{eq:variance_est}. Then with probability at least $1-2\log_{2}{(3 O)}e^{-s} -  (O\vee s)e^{-cN/(O\vee s)}$,
\be
|\wt{\mu}_p(s) - \mu| \leq C_p\sigma \l( \sqrt{\frac{s}{N}} + \sqrt{\frac{O}{N}}\r),
\ee
where $c>0$ depends only on the distribution of $X$ and $C_p>0$ depends only on $p$.
\end{theorem}

\section{Numerical simulation results.}
\label{sec:simul}

The goal of this section is to compare performance of the estimators $\wh{\mu}_{N,p}$ for different values of $p\geq 1$, as well as evaluate their performance against the benchmarks given by other popular techniques such as the median-of-means estimator and the ``oracle'' trimmed mean (labeled ``trim'' in the figures) estimator that takes the contamination proportion $\eps$ as its input.  

Our simulation setup was defined as follows: $N=2500$ observations from half-t distribution\footnote{$X$ has half-t distribution with $\nu$ d.f. if $X=|Y|$ where $Y$ has Student's t-distribution with $\nu$ d.f.} with $4$ degrees of freedom (d.f.). This distribution is asymmetric, therefore, results allow us to evaluate the degree to which the bias affects performance of different robust estimators; linear transformation has been applied so that the mean and variance of generated data are $0$ and $1$ respectively. 
Next, $ O\in \{0,50,100,150\}$ randomly selected observations have been replaced by the outliers given by the point mass at $x_0=10^3$; this type of outliers appears to be most challenging for the trimmed mean estimator as it creates bias due to ``inliers'' being removed only from one of the tails of the distribution. 
We compared 4 estimators: the median-of-means (MOM) estimator defined after equation \eqref{eq:estimator}, estimators $\wh{\mu}_{N,1}$ and $\wh{\mu}_{N,2}$ corresponding to the choice of weights \eqref{eq:weights} with $p=1$ and $p=2$, as well as the ``oracle'' trimmed mean estimator \cite{lugosi2019robust} that knows the number of outliers. Specifically, trimmed mean was computed by removing the smallest $\lfloor \eps N \rfloor + 5$ and well as largest $\lfloor \eps N \rfloor + 5$ observations, where $5$ was added to account for the outliers due to the heavy tails, and averaging over the rest.  
Estimators $\wh{\mu}_{N,1}$, $\wh{\mu}_{N,2}$ as well as MOM were evaluated for various values of parameter $k\in\{25,50,75,100,125,150,175,200\}$ that controls the number of subgroups. 

For each combination of values of $ O$ and $k$, simulation was repeated $1000$ times; we present 3 summary statistics in the plots below: the average error (Figure \ref{fig:average}), the standard deviation (Figure \ref{fig:var}) and the maximal (over 1000 repetitions) absolute error (Figure \ref{fig:max}). 
   \begin{figure}[ht]
        \centering
        \includegraphics[scale=0.25]{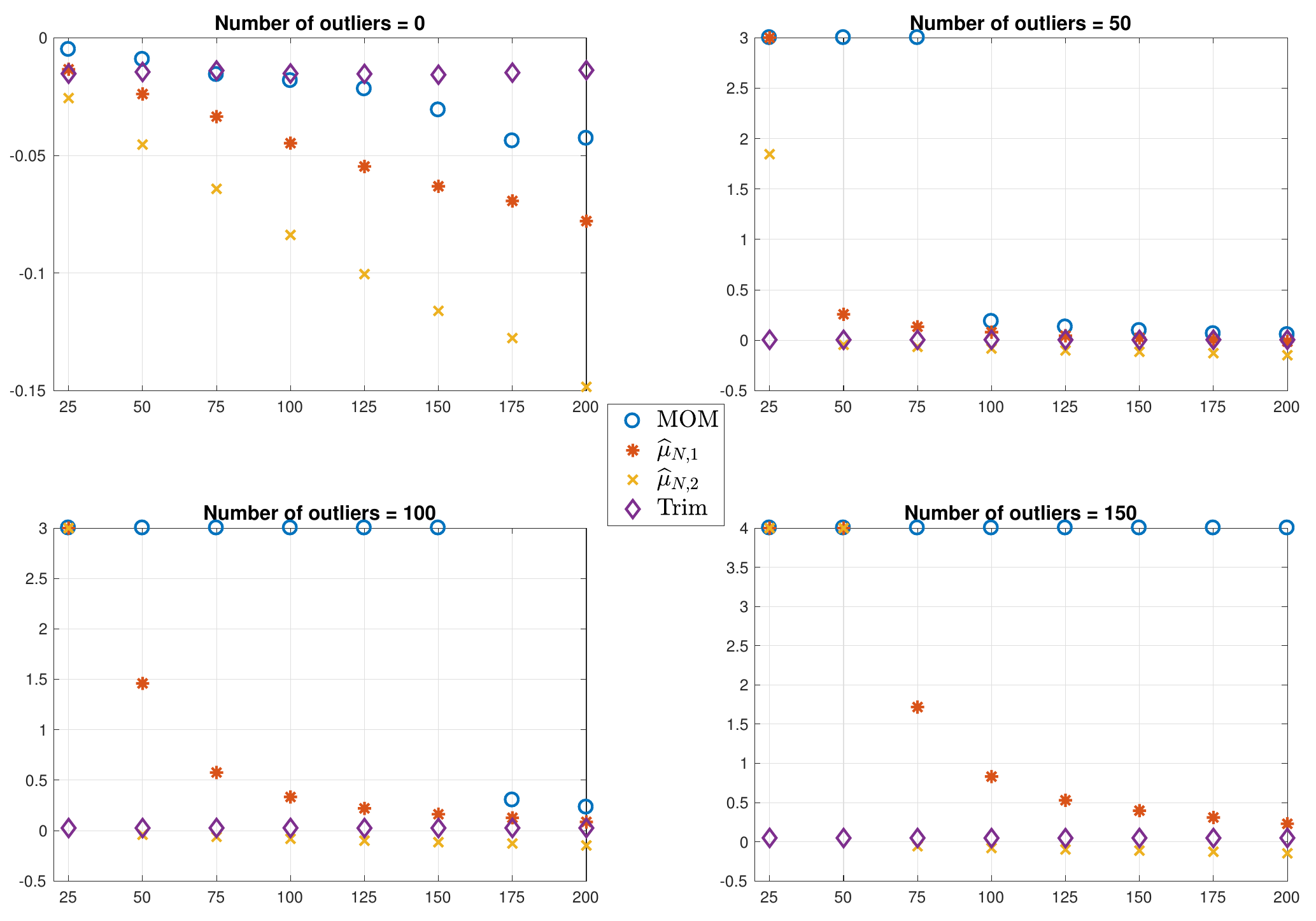}
        \caption{Average estimation error over 1000 runs of the experiment; large values were truncated to show results on appropriate scale.}
           \label{fig:average}
    \end{figure}
   \begin{figure}[ht]
        \centering
        \includegraphics[scale=0.25]{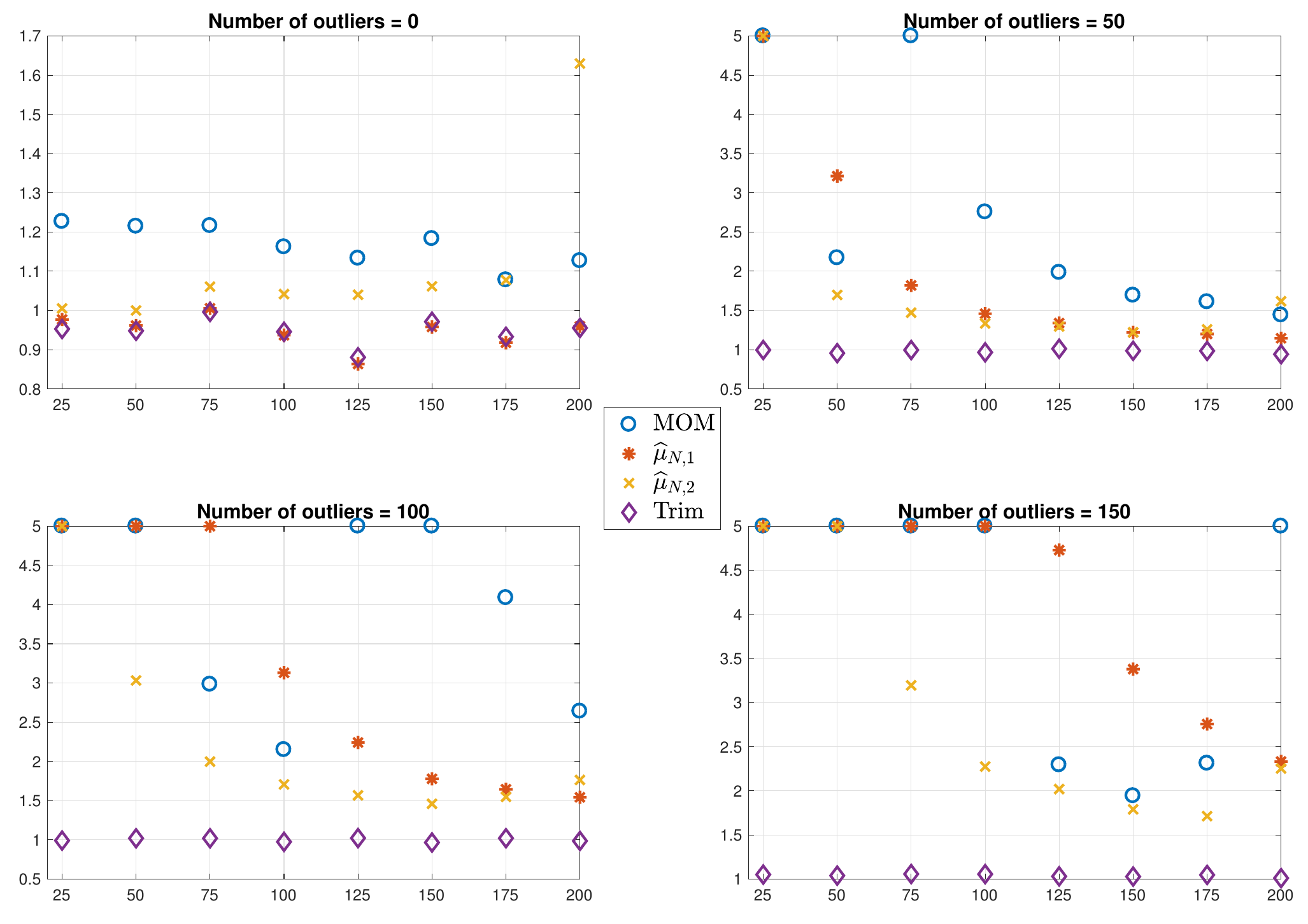}
        \caption{Standard deviation (rescaled by $\sqrt{N}$) over 1000 runs of the experiment; large values were truncated to show results on appropriate scale. }
           \label{fig:var}
    \end{figure}
 \begin{figure}[ht]
        \centering
        \includegraphics[scale=0.25]{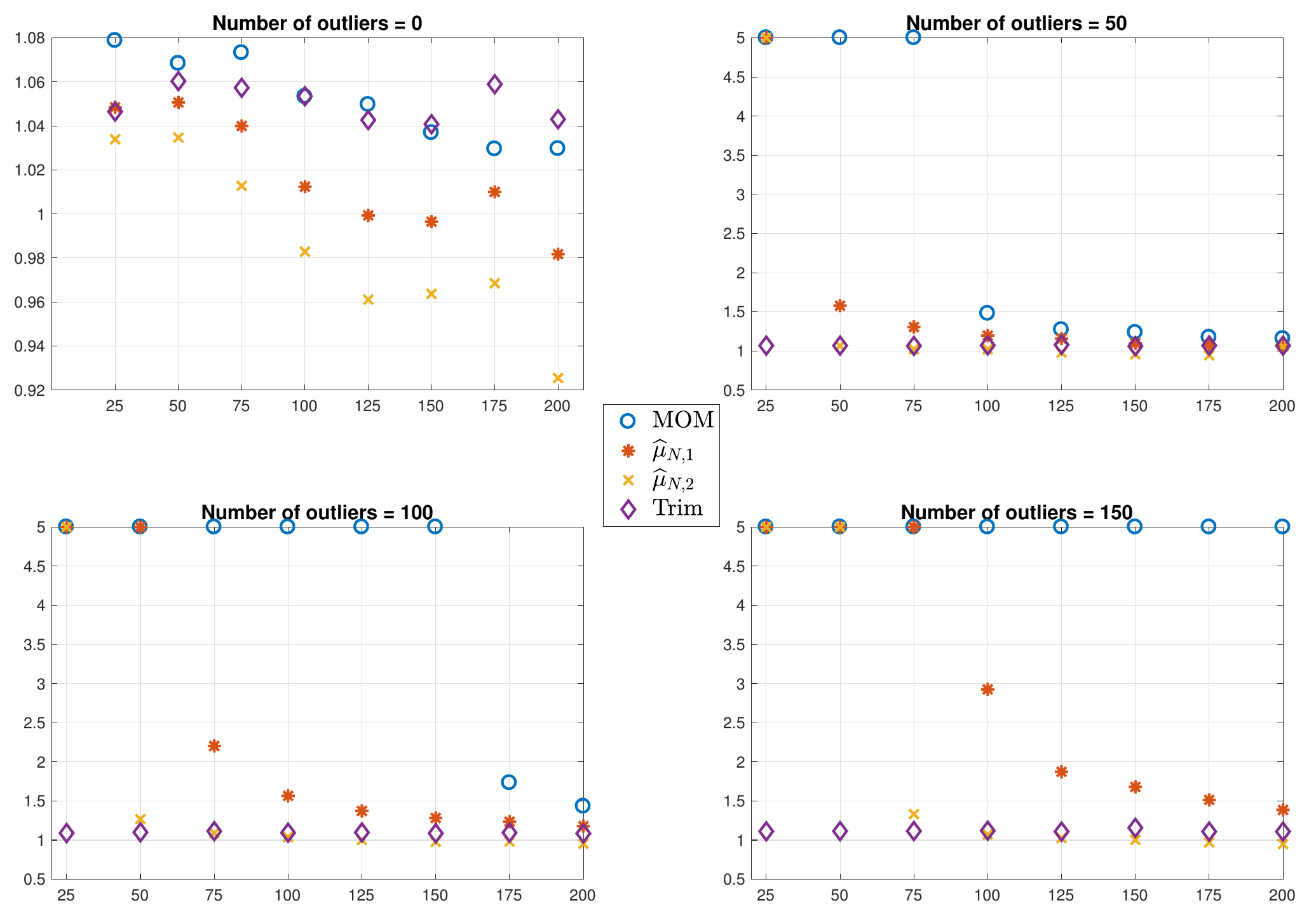}
        \caption{Maximal absolute error over 1000 runs of the experiment; large values were truncated to show results on appropriate scale.}
           \label{fig:max}
    \end{figure}

Overall, numerical experiments confirm our theoretical findings. Here is the summary of our simulation results: 
\begin{enumerate}
\item In the setup with no contamination ($ O=0$), all estimators showed good performance, with $\wh{\mu}_{N,1}$ slightly but consistently beating $\wh{\mu}_{N,2}$ on average, but $\wh{\mu}_{N,2}$ had the smallest maximal error among all estimators; empirical standard deviations of $\wh{\mu}_{N,1}$ and $\wh{\mu}_{N,2}$ were consistent with theory-predicted values;
\item as $ O$ increased, $\wh{\mu}_{N,2}$ was performing better that $\wh{\mu}_{N,1}$, while both estimators were significantly better than MOM;
\item both $\wh{\mu}_{N,1}$ and $\wh{\mu}_{N,2}$ showed consistent performance as the number of blocks $k$ increased; moreover, unlike MOM, the estimators performed well even in the challenging setup where $ O \simeq k$.
\end{enumerate}

\section{Proofs.}
\label{sec:proofs}

This section contains detailed proofs of the main results of the paper.

\subsection{Results related to the deviations of self-normalized sums.}

\subsubsection{Proof of Lemma \ref{lem:sns-2}.}

Let $Z_i:=\frac{X_i -\mu}{\sigma}$, and observe that 
\ml{
\pr{\sum_{i=1}^n Z_{i}^2 \leq \frac{n}{4}}\leq \pr{\sum_{i=1}^n Z_{i}^2 \I\{|Z_{i}| \leq \zeta(X)\} \leq \frac{n}{4}} 
\\
= \pr{\sum_{i=1}^n \l( Z_{i}^2 \I\{|Z_{i}| \leq \zeta(X)\} -\mb EZ_{i}^2 \I\{|Z_{i}| \leq \zeta(X)\}\r) \leq \frac{n}{4} - n H}
\\
\leq \pr{\sum_{i=1}^n \l( Z_{i}^2 I\{|Z_{i}| \leq \zeta(X)\} -\mb EZ_{i}^2 \I\{|Z_{i}| \leq \zeta(X)\}\r) \leq -\frac{n}{4}},
}
where $H=\mb E Z^2 \I\{|Z| \leq \zeta(X)\}$. The last inequality follows from the inequality $H=\mb E Z^2 \I\{|Z| \leq \zeta(X)\} \geq 1/2$ implied by the definition of $\zeta(X)$. The right side of the previous display can be upper bounded via Bernstein's inequality by 
$e^{-\frac{n}{32\l( \zeta^{2}(X)+ \zeta(X)/12 \r)}}$ once we notice that 
\[
\mb E \l(Z^4\mathbf{1}(|Z| \leq \zeta(X) \r) \leq \zeta^{2}(X).
\]
The claim follows from an algebraic inequality $12\zeta^2(X)+ \zeta(X) \leq 15\zeta^2(X) \vee 2$ entailing that 
$e^{-\frac{n}{32\l( \zeta^{2}(X)+ \zeta(X)/12 \r)}} \leq e^{-\frac{n}{40\zeta^2(X) \vee 6}}$.

\subsubsection{Bounds for the moment generating function of the t-statistic.}

Recall that $T_j:= \frac{\bar{\mu}_j - \mu}{\wh{\sigma}_j}$ and $Q_j:=\frac{\bar \mu_j - \mu}{V_j}$ where $V_j^2 = \frac{1}{|G_j|}\sum_{i\in G_j} (X_i - \mu)^2$, $j=1,\ldots,k$. For all $p\geq 1$ define 
\[
w_j = \frac{\sigma^{p-1}T_{j}}{\wh{\sigma}_j^{p-1}}\I\{\mathcal{E}_j\} - \mb E \l( \frac{\sigma^{p-1}T_{j}}{\wh{\sigma}_j^{p-1}}\I\{\mathcal{E}_j\} \r)
\]
where $\m E_j=\{|Q_j| \leq 1/2\}\cap\{V_j \geq \sigma / 2\} $.
\begin{lemma}
\label{lem:concentration2}
There exists $c_p>0$ such that, for all $\lambda \in \mathbf{R}$. we have
\be
\mb E(e^{\lambda w_1}) \leq e^{c_p\lambda^2/(2n)}.
\ee
\end{lemma}
\begin{proof}
We start by observing that on event $\mathcal{E}_1$, $\wh{\sigma}_1 = V_{1}\sqrt{1 - Q_1^2} \geq \frac{\sqrt{3}\sigma}{4}$. 
Hence for all $t>0$, the discussion following \eqref{eq:sns-5} and the inequality $|T_1\mathbf{1}\{\m E_1\}| \leq \frac{\sqrt{3}}{3}$ imply that
\be
\pr{\l|\frac{\sigma^{p-1}T_{1}}{\wh{\sigma}_1^{p-1}}\I\{\m E_1\}\r| \geq t } \leq  
\pr{\l|T_1\I\{\m E_1\}\r| \geq t \l( \frac{\sqrt 3}{4}\r)^{p-1} } \leq 4e^{-c'_p nt^{2}}
\ee
where $c'_p = \frac{1}{2\cdot 11^2}\l( \frac{\sqrt{3}}{4}\r)^{2p-2}$. 
Next, let $\wt w_1$ be an independent copy of $w_1$, and note that 
\ben
\label{eq:sg}
\Pr\l(\l|w_1 - \wt w_1\r| \geq t \r) \leq 8e^{-c'_p nt^{2}/4}.
\een 
It follows from Jensen's inequality that 
\[
\mb E e^{\lambda w_1} \leq  \mb E e^{\lambda (w_1 - \wt w_1)}.
\]
Finally, is well-known that, in view of \eqref{eq:sg}, the latter is bounded by $e^{\lambda^2 c_p/(2n)}$ for some $c_p>0$ only depending on $p$ (for instance, this follows from  Proposition 2.5.2 in \cite{vershynin2018high}). 

\end{proof}

\subsection{Auxiliary technical results.}

\begin{lemma}
\label{lem:bias}
Let $p \geq 1$ and $\delta\geq 1$. Assume that $\mb E(|X-\mu|^{1+\delta}) < \infty$. Then for any $1\leq j\leq k$
\be
\sigma^{p-1}\mb E\l(\frac{\bar{\mu}_j-\mu}{V_j^{p}}\mathbf{1}\{V_j^2 \geq \sigma^2 /4\}\r)=o\left(\frac{1}{n^{\delta/2}}\right),
\ee
for $\delta<2$. At the same time, for $\delta \geq 2$, we have
\be
\sigma^{p-1}\mb E\l(\frac{\bar{\mu}_j-\mu}{V_j^{p}}\mathbf{1}\{V_j^2 \geq \sigma^2 /4\}\r)=O\left(\frac{1}{n}\right).
\ee
\end{lemma}
\begin{proof}
Due to homogeneity, we can assume that $\sigma =1$ without loss of generality. 
We will also assume that $\mu=0$, otherwise $X_j$ should be replaced by $X_j - \mu$ for all $j$. Observe that
\ben\label{eq:bias-1}
\sigma^{p-1}\mb E\l(\frac{\bar{\mu}_j-\mu}{V_j^{p}}\mathbf{1}\{V_j^2 \geq \sigma^2/4\}\r) = n^{p/2}\mb E\left(\frac{X_1}{(X_1^2+Z^2)^{p/2}}\mathbf{1}\{X_1^2+Z^2 \geq n/4\}\right),
\een
where $Z = \sqrt{\sum_{i=2}^n X_i^2}$. Consider the event $\mathcal{O}_1=\{Z^2 \geq n/4\}$, and recall that in view of Lemma $\ref{lem:sns-2}$
\be
\Pr(\mathcal{O}^c_1) \leq e^{-cn}
\ee
for some $c=c(P)>0$ that depends on the distribution $P$ of $X$. 
Consider the event
\[
\mathcal{O}_2 = \{ |X_1| \leq \alpha_n\sqrt{n}/2\},
\] 
where the sequence $\{\alpha_n\}_{n\geq 1}$ is defined as follows: consider a non-increasing function 
$g(u) = \mb E(|X|^{1+\delta}\mathbf{1}_{\{|X|\geq u\}})$, and observe that $\underset{u \to \infty}{\lim}g(u) = 0$. 
Therefore, taking $\alpha_n := g(n^{1/4})^{1/(2+\delta)} \vee n^{-1/4}$, we get that $\alpha_n \to 0$, and moreover
\be
\underset{n \to \infty}{\lim} \frac{g(\alpha_n \sqrt{n})}{\alpha_n^{1+\delta}} \leq \underset{n \to \infty}{\lim} \frac{g(n^{1/4})}{\alpha_n^{1+\delta}} \leq \underset{n \to \infty}{\lim} g^{\frac{1}{2+\delta}}(n^{1/4}) = 0.
\ee
It is easy to see that
\be
\Pr(\mathcal{O}^c_2) = o\left( \frac{1}{n^{(1+\delta)/2}}\right),
\ee
and that
\be
\mb E(|X_{1}|\mathbf{1}_{\{\mathcal{O}_2^c\}}) = o\left(\frac{1}{n^{\delta/2}}\right).
\ee
Indeed, Markov's inequality implies that
\be
\Pr(\mathcal{O}^c_2) 
\leq \frac{2^{1+\delta}\mb E(|X|^{1+\delta}\mathbf{1}_{\{\mathcal{O}^c_2\}})}{\alpha_n^{1+\delta} n^{(1+\delta)/2}} 
= \frac{2^{1+\delta}g(\alpha_n \sqrt{n}/2)}{\alpha_n^{1+\delta} n^{(1+\delta)/2}},
\ee
and 
\ml{
\mb E(|X_{1}|\mathbf{1}_{\{\mathcal{O}_2^c\}})  
\leq \frac{2^{\delta}\mb E (|X|^{1+\delta}\mathbf{1}_{\{\mathcal{O}^c_2\}})}{\alpha_n^{\delta} n^{\delta/2}} 
= \frac{2^{\delta}g(\alpha_n \sqrt{n}/2)}{\alpha_n^{\delta} n^{\delta/2}}
\\
\leq 2^{\delta}\l( g(n^{1/4})\r)^{2/(2+\delta)} n^{-\delta/2},
}
where we used H\"{o}lder's inequality. 
Next, we will reduce the problem to the case where $X$ and $Z$ are bounded. 
Define the event $\wt{\m O} := \mathcal{O}_1 \cap \mathcal{O}_2$. Then $1 = \I_{\wt{\m O}} + \I_{\m O_1\setminus \m O_2 }+ \I_{ \m O_2\setminus \m O_1}$, and 
\mln{
\label{eq:bias-2}
\left|\mb E\left(\frac{X_1}{(X_1^2+Z^2)^{p/2}}\mathbf{1}\{X_1^2+Z^2 \geq n/4\}\right) \right| 
\\
\leq \left|\mb E\left(\frac{X_1}{(X_1^2+Z^2)^{p/2}}\mathbf{1}_{\{ \wt{\m O} \}}\right) \right| + (n/4)^{-p/2}\left(\mb E(|X_1|\mathbf{1}\{\mathcal{O}_2^c\} + \frac{\alpha_n\sqrt{n}}{2} \Pr(\mathcal{O}_1^c)\right)\\
\leq \left|\mb E\left(\frac{X_1}{(X_1^2+Z^2)^{p/2}}\mathbf{1}_{\{ \wt{\m O} \}}\right) \right| + o\l(\frac{1}{n^{(p+\delta)/2}}\r).
}
Letting $F$ be the distribution function of $X$, we deduce that conditionally on $Z$
\mln{
\label{eq:B3}
\left|\mb E\left(\frac{X_1}{(X_1^2+Z^2)^{p/2}}\mathbf{1}_{\{ \wt{\m O} \}}\right)\right| 
=  \left|\int_{-\alpha_n\sqrt{n}/2}^{\alpha_n\sqrt{n}/2} \frac{x}{(x^2+Z^2)^{p/2}}\mathbf{1}_{\{\mathcal{O}_1\}}dF(x)\right|
\\
\leq \left|\int_{-\alpha_n\sqrt{n}/2}^{\alpha_n\sqrt{n}/2} \left(\frac{x}{(x^2+Z^2)^{p/2}}-\frac{x}{Z^p}\right)\mathbf{1}_{\{\mathcal{O}_1\}}dF(x)\right| + \l(\frac{2}{\sqrt{n}}\r)^{p}\mb E \l(|X_{1}|\mathbf{1}_{\{\mathcal{O}_2^c\}} \r)
\\
\leq \left|\int_{-\alpha_n\sqrt{n}/2}^{\alpha_n\sqrt{n}/2} \frac{xZ^p(1-(1+x^2/Z^2)^{p/2})}{(x^2+Z^2)^{p/2}Z^p}\mathbf{1}_{\{\mathcal{O}_1\}}dF(x)\right| + o\l(\frac{1}{n^{(p+\delta)/2}}\r)
\\
\leq \int_{-\alpha_n\sqrt{n}/2}^{\alpha_n\sqrt{n}/2} \frac{p|x|^3}{2Z^{p+2}}\mathbf{1}_{\{\mathcal{O}_1\}}dF(x) + o\l(\frac{1}{n^{(p+\delta)/2}}\r)
\\
 \leq C(p)\mb E|X|^{1+\delta} \frac{\alpha_n^{2-\delta}}{n^{(p+\delta)/2}} + o\l(\frac{1}{n^{(p+\delta)/2}}\r).
}
In the derivation above, we used the elementary inequality 
\ben
\label{eq:elem-ineq}
\frac{(1+t)^{p/2}-1}{(1+t)^{p/2}} = \frac{\int_1^{1+t} \frac{p}{2} y^{p/2 -1} dy}{(1+t)^{p/2}} \leq \frac{pt}{2} \frac{(1+t)^{p/2-1}}{(1+t)^{p/2}}\leq pt/2 
\een
for $0<t:=\frac{x^2}{Z^2}$ and the fact that $\mb E|X|^{1+\delta}<\infty$. 
Combining \eqref{eq:B3} with \eqref{eq:bias-1},\eqref{eq:bias-2}, we see that 
\be
\sigma^{p-1}\mb E\l(\frac{\bar{\mu}_j-\mu}{V_j^{p}}\mathbf{1}\{V_j^2 \geq \sigma^2/4\}\r) = o\left( \frac{1}{n^{\delta/2}}\right)
\ee
whenever $\delta<2$ and that
\be
\sigma^{p-1}\mb E\l(\frac{\bar{\mu}_j-\mu}{V_j^{p}}\mathbf{1}\{V_j^2 \geq \sigma^2/4\}\r) = O\left( \frac{1}{n}\right),
\ee
for $\delta = 2$ (in fact, in this case all the terms are of order $o\l(n^{-1}\r)$ besides $ C(p)n^{p/2}\mb E|X|^{1+\delta} \frac{\alpha_n^{2-\delta}}{n^{(p+\delta)/2}}$ which is $O\l(n^{-1}\r)$).
\end{proof}
\begin{remark}
It follows from the previous argument that the term $o\l( \frac{1}{n^{\delta/2}}\r)$ takes the form
\[
n^{-\delta/2}\cdot \l( n^{-\frac{2-\delta}{4}} \vee g^{\frac{2-\delta}{2+\delta}}(n^{1/4})\r).
\]
\end{remark}
\begin{remark}
The key quantity of interest in the previous proof is given by the expression
\be
\left|\mb E\left(\frac{X_1}{(X_1^2+Z^2)^{p/2}}\mathbf{1}\{X_1^2+Z^2 \geq n/4\}\right) \right|
\ee 
that was then estimated from above. 
Let us present a counterexample showing that one cannot improve the result of Lemma \ref{lem:bias} when $\delta \geq 2$ for $p=1$. To this end, let $X$ be a random variable such that $\Pr(X = a) = 1/(1+a^2) $ and $\Pr(X = -1/a) = a^2/(a^2+1)$ for some $1 < a^2 \leq 2 $ and assume that $n \geq 8$. 
Observe that $X$ is a.s. bounded by $a$, centered, and has variance $1$. 

Given $x,y>0$, we say that $x \asymp y$ when $c \leq x/y \leq C$ for some absolute constants $c,C>0$. Let $\mb E_Z$ denote the conditional expectation with respect to $Z$.  It is easy to check that on the event $\mathcal{A}:=\{Z^2 \geq n/4\}$ we have
\ml{
\mb E_Z\left(\frac{X_1}{\sqrt{X_1^2 + Z^2}}\mathbf{1}\{X_1^2+Z^2 \geq n/4\} \I\{\m A\} \right) = \mb E_Z\left(\frac{X_1}{\sqrt{X_1^2 + Z^2}} \right)\I\{\m A\} \\
=\frac{a}{\sqrt{a^2 + Z^2}}\frac{1}{1+a^2}\I\{\m A\}  - \frac{1}{a\sqrt{1/a^2 + Z^2}}\frac{a^2}{1+a^2}\I\{\m A\} 
\\
= \frac{a}{1+a^2}\frac{\sqrt{1/a^2 + Z^2}-\sqrt{a^2 + Z^2}}{\sqrt{a^2 + Z^2}\sqrt{1/a^2 + Z^2}}\I\{\m A\} \\
= \frac{a}{1+a^2}\frac{1/a^2 -a^2 }{\sqrt{a^2 + Z^2}\sqrt{1/a^2 + Z^2}(\sqrt{1/a^2 + Z^2}+\sqrt{a^2 + Z^2})}\I\{\m A\} 
\\
\asymp \frac{a}{1+a^2}\frac{a^2 - 1/a^2}{Z^{3}}\I\{\m A\} 
\asymp \frac{1}{n^{3/2}}\I\{\m A\} ,
}
where we have used that on $\m A$ both $a^2$ and $1/a^2$ are smaller than $Z^2$ and that $Z^2 \asymp n$. Since $a$ does not depend on $n$, $X$ is a.s. bounded by an absolute constant, and $\Pr(  \m A) \geq 1-e^{-cn}$ for some absolute constant $c>0$. Hence
\be 
\mb E\left(\frac{X_1}{\sqrt{X_1^2 + Z^2}}\mathbf{1}\{X_1^2+Z^2 \geq n/4\}  \right) \asymp \frac{1}{n^{3/2}}\Pr(\m A) + \Pr(\m A^c) \asymp \frac{1}{n^{3/2}}.
\ee

It follows that, for $p=1$, we have
\ml{
\mb E\l(\frac{\bar{\mu}_j-\mu}{V_j}\mathbf{1}\{V_j^2 \geq \sigma^2/4\}\r) 
\\ = n^{1/2}\mb E\left(\frac{X_1}{(X_1^2+Z^2)^{1/2}}\mathbf{1}\{X_1^2+Z^2 \geq n/4\}\right)\asymp n^{-1}.
}
Although $X$ admits infinitely many moments, the previous bound cannot be improved beyond three moments due to the asymmetry of the distribution of $X$.
\end{remark}

\begin{lemma}
\label{lem:variance}
Let $p\geq1$. If $\var(X) < \infty$, then
\be
\underset{n \to \infty}{\lim}\sigma^{2p-2}n\,\mb E\l(\frac{\bar{\mu}_1-\mu}{V_1^{p}}\r)^2\mathbf{1}\{V_1^2 \geq \sigma^2 /4\} = 1.
\ee
\end{lemma}
\begin{proof}
Again, we can assume without loss of generality that $\sigma^2=1$ and that $\mb EX = 0$. Observe that
\ml{
n\,\mb E\l(\frac{\bar{\mu}_1 - \mu}{V_1^{p}}\r)^2 \mathbf{1}\{V_1^2 \geq \sigma^2 /4\} = n^{p}\mb E\left(\frac{X_1^2}{(X_1^2+Z^2)^{p}}\mathbf{1}\{X_1^2+ Z^2 \geq n /4\}\right) 
\\
+ n^{p}\mb E\left(\frac{X_1Y}{(X_1^2+Z^2)^{p}}\mathbf{1}\{X_1^2 + Z^2 \geq n /4\}\right),
}
where $Y = \sum_{i=2}^nX_i$ and $Z = \sqrt{\sum_{i=2}^n X_i^2}$. 
It is clear that
\be
n^{p}\frac{X_1^2}{(X_1^2+Z^2)^{p}} \I\{X_1^2 + Z^2 \geq n /4\} \to X_1^2 \text{ in probability.}
\ee
Indeed, $\I\{X_1^2 + Z^2 \geq n /4 \} \to 1$ in probability in view of Lemma \ref{lem:sns-2}, while $\l(\frac{X_1^2 + Z^2}{n}\r)^p \to 1$ in probability by the Law of Large Numbers. 

Moreover, $n^{p}\frac{X_1^2}{(X_1^2+Z^2)^{p}} \I\{X_1^2 + Z^2 \geq n /4\} \leq 4^{p}X_1^2$. Therefore 
\be
n^{p}\mb E\left(\frac{X_1^2}{(X_1^2+Z^2)^{p}}\mathbf{1}\{X_1^2 + Z^2 \geq n /4\}\right) \xrightarrow{n\to\infty} 1.
\ee
It remains to prove that
\be
n^{p}\mb E\left(\frac{X_1Y}{(X_1^2+Z^2)^{p}}\mathbf{1}\{X_1^2 + Z^2 \geq n /4\}\right) \xrightarrow{n\to\infty} 0.
\ee
Consider the event $\mathcal{O}_1=\{Z^2 \geq n/4\}$, and recall that 
$\Pr(\mathcal{O}^c_1) \leq e^{-cn}$ for some $c>0$ that depends on the distribution of $X$ as given in Lemma $\ref{lem:sns-2}$. 
We will also need to consider the event 
\[
\mathcal{O}_2 = \{ |X_1| \leq \alpha_n\sqrt{n}/2\}
\] 
where $(\alpha_n)_n$ is defined as in Lemma \ref{lem:bias} with $\delta=1$. Namely, consider the non-increasing function $g(u) = \mb E(|X|^{2}\mathbf{1}_{\{|X|\geq u\}})$, and define $\alpha_n = g(n^{1/4})^{1/3} \vee n^{-1/4}$, so that $\alpha_n \to 0$ and 
\be
\underset{n \to \infty}{\lim} \frac{g(\alpha_n \sqrt{n})}{\alpha_n^{2}} \leq \underset{n \to \infty}{\lim} \frac{g(n^{1/4})}{\alpha_n^{2}} =0.
\ee
As in the proof of Lemma \ref{lem:bias}, we deduce that $\Pr(\mathcal{O}^c_2) = o\left( \frac{1}{n}\right)$ and that
\ben
\label{eq:o-small}
\mb E(|X_{1}|\mathbf{1}_{\{\mathcal{O}_2^c\}}) = o\left(\frac{1}{\sqrt{n}}\right).
\een
Next, we will reduce the problem to the case where $X$ and $Z$ are bounded. Let $\wt{\m O}:= \mathcal{O}_1 \cap \mathcal{O}_2$. Then
\mln{
\label{eq:b1}
n^{p} \left|\mb E\frac{X_1Y}{(X_1^2+Z^2)^{p}}\mathbf{1}\{X_1^2 + Z^2 \geq n /4\}\right|
\\
\leq n^p\left|\mb E\left(\frac{X_1Y}{(X_1^2+Z^2)^{p}}\mathbf{1}_{\{\wt{\m O}\}}\right) \right| + c_1(p) \sqrt{n}\mb E(|X_1|\mathbf{1}_{\mathcal{O}_2^c}) 
\\ +c_2(p) n \l(\Pr(\mathcal{O}_1^c)\r)^{1/2}.
}
Indeed, $1 = \I_{\wt{\m O}} + \I_{\m O_1\setminus \m O_2 }+ \I_{ \m O_2\setminus \m O_1}$, and 
\ml{
\l| \mb E\frac{X_1Y}{(X_1^2+Z^2)^{p}} \I\{X_1^2 + Z^2 \geq n /4\} \I_{\{ \m O_1\setminus \m O_2 \}} \right| 
\\
\leq \l(\frac{4}{n}\r)^{p-1/2}\mb E \l| X_1\I_{\{ \m O_2^c\}} \frac{2Y}{\sqrt n} \r| 
\\
= \l( \frac{4}{n}\r)^{p-1/2} \mb E \l| X_1\I_{\{ \m O_2^c\}} \r| \mb E \l|\frac{2Y}{\sqrt n} \r| 
\\
\leq 2\l(\frac{4}{n}\r)^{p-1/2} \mb E \l| X_1\I_{\{ \m O_2^c\}} \r|:=\frac{c_1(p)}{n^{p-1/2}} \mb E \l| X_1\I_{\{ \m O_2^c\}} \r|
}
as $\mb E \l|\frac{Y}{\sqrt n} \r| \leq \frac{\mb E^{1/2} Y^2}{\sqrt n} \leq 1$. Moreover, 
\ml{
\l| \mb E\frac{X_1Y}{(X_1^2+Z^2)^{p}} \I\{X_1^2 + Z^2 \geq n /4\} \I_{\{ \m O_2 \setminus \m O_1\}} \right| 
\leq \l( \frac{4}{n}\r)^{p-1/2} \frac{\alpha_n\sqrt{n}}{2} \mb E \l|\frac{2Y}{\sqrt n} \I_{\m O_1^c}\r| 
\\
\leq 2\l( \frac{4}{n}\r)^{p-1/2} \frac{\alpha_n\sqrt{n}}{2} \l( \pr{\m O_1^c} \r)^{1/2} := \frac{c_2(p)}{n^{p-1}} \l( \pr{\m O_1^c} \r)^{1/2},
}
thus \eqref{eq:b1} follows.
Next, letting $F$ be the distribution function of $X$, we deduce that conditionally on $(Y,Z)$,
\ml{
\left|\mb E\left[\frac{X_1Y}{(X_1^2+Z^2)^{p}}\mathbf{1}_{\{ \wt{\m O} \}} \big| Y,Z\right]\right| 
=  \left| \int_{-\alpha_n\sqrt{n}/2}^{\alpha_n\sqrt{n}/2} \frac{xY}{(x^2+Z^2)^{p}}\mathbf{1}_{\{\mathcal{O}_1\}}dF(x)\right|
\\
= \left|\int_{-\alpha_n\sqrt{n}/2}^{\alpha_n\sqrt{n}/2} \left(\frac{xY}{(x^2+Z^2)^{p}}-\frac{xY}{Z^{2p}}\right)\mathbf{1}_{\{\mathcal{O}_1\}}dF(x)\right| 
\\
+\l(\frac{4}{n}\r)^{p-1/2} \frac{|Y| \I_{\m O_1}}{Z}\mb E(|X_{1}|\mathbf{1}_{\{\mathcal{O}_2^c\}})
\\
\leq \left|\int_{-\alpha_n\sqrt{n}/2}^{\alpha_n\sqrt{n}/2} \frac{xYZ^{2p}(1-(1+x^2/Z^2)^{p})}{(x^2+Z^2)^{p}Z^{2p}}\mathbf{1}_{\{\mathcal{O}_1\}} dF(x)\right| 
\\
+ \l(\frac{4}{n}\r)^{p-1/2} \frac{|Y| \I_{\m O_1}}{Z}\mb E(|X_{1}|\mathbf{1}_{\{\mathcal{O}_2^c\}})
\\
\leq \int_{-\alpha_n\sqrt{n}/2}^{\alpha_n\sqrt{n}/2} \frac{p|x|^3|Y|}{Z^{2p+2}}\mathbf{1}_{\{\mathcal{O}_1\}}dF(x)
\\+ 
\l(\frac{4}{n}\r)^{p-1/2} \frac{|Y| \I_{\m O_1}}{Z}\mb E(|X_{1}|\mathbf{1}_{\{\mathcal{O}_2^c\}}).
}
In the derivation above, we used the bound 
\ml{
\l| \int_{-\alpha_n\sqrt{n}/2}^{\alpha_n\sqrt{n}/2} \frac{xY}{Z^{2p}} \I_{\{\mathcal{O}_1\} }dF(x) \r| 
\\
\leq \l(\frac{4}{n}\r)^{p-1/2} \frac{|Y| \I_{\m O_1}}{Z} \l| \int_\mb R x dF(x) - \int_{-\alpha_n\sqrt{n}/2}^{\alpha_n\sqrt{n}/2} x dF(x)\r|
\\
= \l(\frac{4}{n}\r)^{p-1/2} \frac{|Y| \I_{\m O_1}}{Z} \l| \mb E X_1 \I_{\m O_2^c} \r|
}
and relation
\[
\l|\frac{x Y Z^{2p}(1-(1+x^2/Z^2)^{p})}{(x^2+Z^2)^{p}Z^{2p}} \r|
= \l|\frac{x Y (1-(1+x^2/Z^2)^{p})}{(x^2/Z^2+1)^{p} Z^{2p}} \r|
\leq \frac{p|x|^3|Y|}{Z^{2p+2}},
\]
where the last inequality follows from an elementary bound \eqref{eq:elem-ineq}. 
Moreover, 
\ml{
\mb E \l(\int_{-\alpha_n\sqrt{n}/2}^{\alpha_n\sqrt{n}/2} \frac{p|x|^3|Y|}{Z^{2p+2}}\mathbf{1}_{\{\mathcal{O}_1\}}dF(x) \r)
\\
\leq 2p \l(\frac{4}{n}\r)^{p+1/2} \mb E\l(\frac{|Y|}{\sqrt n}\r) \int_{-\alpha_n\sqrt{n}/2}^{\alpha_n\sqrt{n}/2} |x|^3 dF(x)
\\
\leq 4p\alpha_n \l(\frac{4}{n}\r)^{p}\int_\mb R x^2 dF(x) = o\l( \frac{1}{n^p}\r)
}
and 
\ml{
\mb E\l( \l(\frac{4}{n}\r)^{p-1/2} \frac{|Y| \I_{\m O_1}}{Z} \r) \mb E\l(|X_{1}|\mathbf{1}_{\{\mathcal{O}_2^c\}} \r) 
\\
\leq \l(\frac{4}{n}\r)^{p} \mb E\l(\frac{|Y|}{\sqrt n}\r) \sqrt{n}\mb E(|X_{1}|\mathbf{1}_{\{\mathcal{O}_2^c\}}) 
= o\l( \frac{1}{n^p}\r)
}
in view of \eqref{eq:o-small}. Therefore, we see that
\[
\underset{n \to \infty}{\lim}n^{p}\left|\mb E\frac{X_1Y}{(X_1^2+Z^2)^{p}}\mathbf{1}\{X_1^2 + Z^2 \geq n /4\}\right| = 0,
\]
concluding the proof.
\end{proof}

\begin{lemma}\label{lem:bias__variance_student}
Let $p\geq1$, assume that $\mb E |X-\mu|^{1+\delta} < \infty$ for some $\delta\geq 1$. Consider the event $\wt{\m O}=\{|Q_{1}|\leq 1/2 \} \cap \{ V_{1} \geq \sigma/2\}$. Then
\be
\sigma^{p-1}\l|\mb E\l(\frac{\bar{\mu}_1 - \mu}{\wh\sigma_1^{p}}\mathbf{1}\{ \wt{\m O} \}\r)\r| \leq \phi(\delta,n) + 2^{p-1} \sqrt{\frac{ke^{-cn}}{N}},
\ee
 where $c>0$ depends only on $\zeta(X)$, $\phi(\delta,n) = o(n^{-\delta/2})$ for $\delta<2$ and $\phi(\delta,n) = O(n^{-1})$ otherwise. 
Moreover, if $\var(X)<\infty$, then
\be
\var\l(\frac{\sqrt{n}\sigma^{p-1}(\bar{\mu}_1 - \mu)}{\wh\sigma_1^{p}}\mathbf{1}\{\wt{\m O}\}\r)  \xrightarrow{n \to \infty} 1.
\ee
\end{lemma}
\begin{proof} 
We will prove the two claims separately. 
Recall the algebraic identity $\wh{\sigma}_1 = V_1\sqrt{1-Q_1^2}$. To deduce the first inequality, observe that
\ml{
\l|\frac{\bar{\mu}_1-\mu}{\wh\sigma_1^{p}} - \frac{\bar{\mu}_1-\mu}{V_1^{p}}\r|\mathbf{1}\{\wt{\m O}\} 
\\
= \frac{|\bar{\mu}_1-\mu|}{V_1^{p}}\l| (1-Q_1^2)^{-p/2} -1 \r|\mathbf{1}\{ \wt{\m O} \} \leq p4^{p}\sigma^{1-p}|Q_1|^{3}\mathbf{1}\{V_1^2 \geq \sigma^2 /4\} ,
}
where we have used the elementary inequality 
\ben\label{eq:elementary_3}
(1 - t^2)^{-p/2} - 1 
= \l(1+\frac{t^2}{1-t^2}\r)^{p/2}-1 \leq \int_{0}^{\frac{4}{3} t^2}\frac{p}{2}(1+u)^{p/2-1}du 
\leq \frac{2^p p}{3^{p/2}}t^2
\een
that holds for all $0 \leq t \leq 1/2$.
Taking \eqref{eq:sns-2} into account, we get that
\be
\mb E(|Q_1|^{3}\mathbf{1}\{V_1^2 \geq \sigma^2 /4\}) \leq \frac{C}{n^{3/2}}
\ee
for an absolute constant $C>0$. Indeed, it directly follows from the inequality
\ben\label{eq:elementary_2}
\Pr\l( |Q_1|\mathbf{1}\{V_1^2 \geq \sigma^2 /4\}\geq \frac{9x}{\sqrt{n}} \r) \leq 4e^{-x^2}
\een
that is valid for all $x \geq 0$. As a consequence,
\ben\label{eq:asymp_bias_1} 
\l|\mb E\l(\frac{\bar{\mu}_1-\mu}{\wh\sigma_1^{p}}\mathbf{1}\{\wt{\m O}\}\r)\r| \leq \l|\mb E\l( \frac{\bar{\mu}_1-\mu}{V_1^{p}}\mathbf{1}\{\wt{\m O}\}\r) \r| + C\frac{p 4^{p}\sigma^{1-p}}{n^{3/2}}.
\een
Moreover, we have that
\mln{\label{eq:asymp_bias_12}
\l|\mb E\l(\frac{\bar{\mu}_1-\mu}{V_1^{p}}\mathbf{1}\{\wt{\m O}\}\r)\r| 
\leq \l|\mb E\l(\frac{\bar{\mu}_1-\mu}{V_1^{p}}\mathbf{1}\{V_1^2 \geq \sigma^2 /4\}\r)\r|\\
+ 2^{p-1}\sigma^{1-p}\mb E(|Q_1|\mathbf{1}\{|Q_1|\geq 1/2\}\cap\{V_1^2 \geq \sigma^2 /4\})  \\
\leq \l|\mb E\l(\frac{\bar{\mu}_1 - \mu}{V_1^{p}}\mathbf{1}\{V_1^2 \geq \sigma^2 /4\}\r)\r| \\
+2^{p-1}\sigma^{1-p}\sqrt{\mb E(Q_1^2\mathbf{1}\{V_1^2 \geq \sigma^2 /4\})\Pr(\{|Q_1|\geq 1/2\}\cap\{V_1^2 \geq \sigma^2 /4\})}  
\\
\leq \l|\mb E\l(\frac{\bar{\mu}_1 - \mu}{V_1^{p}}\mathbf{1}\{V_1^2 \geq \sigma^2 /4\}\r)\r| + 2^{p-1}\sigma^{1-p} \sqrt{\frac{ke^{-cn}}{N}},
}
where we have used \eqref{eq:elementary_2} in the last inequality. We conclude using Lemma \ref{lem:bias} as long as \eqref{eq:asymp_bias_1} and \eqref{eq:asymp_bias_12} that
\ben\label{eq:asymp_bias_2}
\l|\mb E\l(\frac{\bar{\mu}_1-\mu}{\wh\sigma_1^{p}}\mathbf{1}\{\wt{\m O}\}\r)\r| \leq \sigma^{1-p}\l(\phi(\delta,n)+ 2^{p-1} \sqrt{\frac{ke^{-cn}}{N}} \r).
\een
The first claim is a consequence of both \eqref{eq:asymp_bias_1} and \eqref{eq:asymp_bias_2} since $n^{-3/2}$ is always less than $\phi(\delta,n)$.

Next, we establish the second claim of the lemma. 
Since, due to the first inequality of the lemma, $\sqrt{n}\l|\mb E \frac{\sigma^{p-1}(\bar{\mu}_1 - \mu)}{\wh\sigma_1^{p}}\mathbf{1}\{\wt{\m O}\}\r|$ vanishes as $n$ goes to infinity, it is enough to prove that the second moment converges to $1$. 
We follow the same steps as in the first part to deduce that 
\mln{
\label{eq:asymp_var_1}
n\sigma^{2p-2}\mb E\l|\l(\frac{\bar{\mu}_1-\mu}{\wh\sigma_1^{p}}\r)^2 - \l(\frac{\bar{\mu}_1-\mu}{V_1^{p}}\r)^2\r|\mathbf{1}\{\wt{ \m O}\} 
\\
= n\sigma^{2p-2}\frac{|\bar{\mu}_1-\mu|^2}{V_1^{2p}}\l| (1-Q_1^2)^{-p} -1 \r|\mathbf{1}\{ \wt{\m O} \} \\
\leq 2p8^{p} n\mb E\l(|Q_1|^{4}\mathbf{1}\{V_1^2 \geq \sigma^2 /4\}\r) \leq \frac{pC8^p}{n},
}
where we have used \eqref{eq:elementary_3} in the first inequality and \eqref{eq:elementary_2} in the second one.
Moreover, we also have that
\mln{
\label{eq:asymp_var_2}
    n\sigma^{2p-2}\mb E\l|\l(\frac{\bar{\mu}_j-\mu}{V_j^{p}}\r)^2\mathbf{1}\{ \wt{\m O}\} - \l(\frac{\bar{\mu}_j-\mu}{V_j^{p}}\r)^2\mathbf{1}\{V_j^2 \geq \sigma^2 /4\}\r|  
    \\
    \leq n4^{p-1}\mb E(Q_1^2\mathbf{1}\{|Q_1|\geq 1/2\}\cap\{V_1^2 \geq \sigma^2 /4\})  
    \\
\leq n4^{p-1}\sqrt{\mb E(Q_1^4\mathbf{1}\{V_1^2 \geq \sigma^2 /4\})\Pr(\{|Q_1|\geq 1/2\}\cap\{V_1^2 \geq \sigma^2 /4\})} 
   \\
   \leq 4^{p-1} e^{-c'n},
}
where we again used \eqref{eq:elementary_2}.  Combining \eqref{eq:asymp_var_1} and \eqref{eq:asymp_var_2}, we get that
\ml{
\underset{n \to \infty}{\lim}\var\l(\frac{\sqrt{n}\sigma^{p-1}(\bar{\mu}_1 - \mu)}{\wh\sigma_1^{p}}\mathbf{1}\{\wt{\m O}\}\r)  \\ = \underset{n \to \infty}{\lim} n\sigma^{2p-2}\mb E\l(\l(\frac{\bar{\mu}_j-\mu}{V_j^{p}}\r)^2\mathbf{1}\{V_j^2 \geq \sigma^2 /4\}\r) .
}
The conclusion follows immediately from Lemma \ref{lem:variance}.
\end{proof}

\begin{lemma}
\label{lemma:minmax}
In the framework of section \ref{sec:intro}, 
\[
\inf_{\wt \mu} \sup_{P\in\m P_{2,\sigma}} \nu^2(\wt\mu,P) = \sigma^2.
\]
\end{lemma}
\begin{proof}
Let $\wt{\m P}$ be the  family of normal distributions $\l\{ N(\mu,\sigma^2), \ \mu\in \mb R\r\}$. 
Then we deduce from the almost-everywhere convolution theorem (Theorem 8.9 in \cite{van2000asymptotic}) that for any $\wt\mu$, $\sup_{P\in\m P_{2,\sigma}} \nu^2(\wt\mu,P) \geq \sigma^2$. 
On the other hand, letting $\wt\mu$ be the sample mean $\wt\mu = \frac{1}{N}\sum_{j=1}^N X_j$, we obtain the reverse inequality $\inf_{\wt \mu} \sup_{P\in\m P_{2,\sigma}} \nu^2(\wt\mu,P) \leq \sigma^2$. 
\end{proof}

\subsection{Proof of Lemma \ref{lem:denominator}.}\label{proof:lem:denominator}

We will first consider the outlier-free case, meaning that $O=0$. It is easy to see that  
\be
\l(\frac{1}{k}\sum_{j=1}^k \frac{1}{\wh \sigma^p_j} \r)^{-1} 
\leq 2\,\med\l( \wh\sigma^p_1,\ldots,\wh\sigma^p_k\r).
\ee
Hence, Bennett's inequality yields that
\mln{
\label{eq:base-ineq}
\Pr\l(\frac{1}{k}\sum_{j=1}^k \frac{1}{\wh \sigma^p_j}  \leq \frac{1}{(4\sigma)^{p}}\r) 
\leq  \Pr\l(\med\l( \wh\sigma_1,\ldots,\wh\sigma_k\r) \geq 2\sigma \r)
\\
\leq \Pr\l(\sum_{j=1}^k(\mathbf{1}\{\wh \sigma_{j}^2 \geq 4\sigma^2\} - \pi) \geq k/4 \r) 
 \leq e^{-ck\l(\log{\frac{1}{\pi}}+1\r)},
}
for some absolute constant $c>0$, where $\pi := \pr{\wh{\sigma}_1^2 \geq 4\sigma^2} \leq  \pr{V_1^2 \geq 4\sigma^2}\leq \frac{1}{4}$. 
Alternatively, if $X$ possesses more than 2 moments, we can apply von Bahr-Esseen inequality \cite{essen} to deduce that
\be
\pi \leq \frac{\mb E|X-\mu|^{1+\delta}/\sigma^{1+\delta}}{n^{\frac{\delta-1}{2}} },
\ee 
for any $\delta \geq 1$. It yields that 
\ben
\label{eq:simple_case_one}
\Pr\l(\frac{1}{k}\sum_{j=1}^k \frac{1}{\wh \sigma^p_j} \leq \frac{1}{(4\sigma)^p} \r) \leq e^{-c'k(1+(\delta-1)\log{n})} ,
\een
for $c'>0$ depending only on the ratio $\mb E|X-\mu|^{1+\delta}/\sigma^{1+\delta}$.
When $X$ has sub-Gaussian distribution, we instead use the Hanson-Wright inequality \cite{hanson1971bound} and deduce that 
\be
\pi \leq  e^{-cn \frac{\sigma^4}{\|X\|^4_{\psi_2}}}
\ee
where $c>0$ is an absolute constant and $\|X\|_{\psi_2}$ is the $\psi_2$ norm of $X$ \footnote{The $\psi_2$ norm of $X$ is defined via $\|X\|_{\psi_2}:=\inf\l\{ C>0: \ \mb E \exp\l( |X/C|^2\r)\leq 2\r\}$.}. 
In this case, \eqref{eq:base-ineq} yields that
\ben\label{eq:simple_case_two_2}
\Pr\l(\frac{1}{k}\sum_{j=1}^k \frac{1}{\wh \sigma^p_j}  \leq \frac{1}{(4\sigma)^{p}}\r) 
\leq e^{-c(P)kn} \leq e^{-\wt c(P)N}
\een
$c(P) := c_1 \frac{\sigma^4}{\|X\|^4_{\psi_2}}$ for an absolute constant $c_1>0$. 

Next, we consider the case $O>0$. 
Let $\wh\sigma_{(1)},\ldots,\wh\sigma_{(k)}$ be the increasing order statistics corresponding to $\wh\sigma_1,\ldots,\wh\sigma_k$. 
If $ O\leq C k$ for $C<1$, then at least a fraction of data buckets is outlier-free. 
Let us call the index set of these buckets $J$ so that $\card(J)\geq \lfloor (1-C)k \rfloor$, whence
\be
\frac{1}{k}\sum_{i=1}^k \frac{1}{\wh{\sigma}^p_i} 
\geq \frac{\lfloor \frac{(1-C)k}{2}\rfloor}{k} \frac{1}{ \wh{\sigma}^p_{\l( \lfloor (1-C)k/2)\rfloor \r)} }. 
\ee
Hence,  we get that
\be
\Pr\l( \frac{1}{k}\sum_{i=1}^k\frac{1}{\wh{\sigma}^p_i} \leq \l(\frac{1-C}{4\sigma}\r)^p \r) \leq  \Pr\l(\wh{\sigma}^p_{\l( \lfloor (1-C)k/2)\rfloor \r)} \geq 2\sigma \r).
\ee
The final result follows from \eqref{eq:simple_case_one} and \eqref{eq:simple_case_two_2} replacing $k$ by $\lfloor (1-C)k \rfloor$.
\begin{lemma}
\label{lem:asymptotic-variance}
Let $\wh{\sigma}_n$ be such that $\wh{\sigma}_n = V_n\sqrt{1-Q_n^2}$, and let $\wt{\m O}=\{|Q_n| \leq 1/2\}\cap\{V_n^2 \geq \sigma^2 / 4\} $ using previous notations. Then 
\be
\underset{n \to \infty}{\lim} \mb E\left|\frac{\sigma^p}{\wh{\sigma}^p_n}\mathbf{1}\{\wt{\m O}\}-1\right| = 0.
\ee
\end{lemma}
\begin{proof}
We have that $Q_n \leq 1/2$ and $\wh{\sigma}^2_n \geq 3/4V^2_n \geq 3/16 \sigma^2$ on $\wt{\m O}$. 
Therefore,
\be
\mb E\left|\frac{\sigma^p}{\wh{\sigma}^p_n}\mathbf{1}\{\wt{\m O}\}-1\right| \leq  \mb E\left|\frac{\wh{\sigma}^p_n - \sigma^p}{\wh{\sigma}_n^p}\right|\mathbf{1}\{\wt{\m O}\} + \Pr(\wt{\m O}^c)\\
 \leq  c_p\mb E\left|\frac{\wh{\sigma}_n - \sigma}{\wh{\sigma}_n}\right|\mathbf{1}\{\wt{\m O}\} + \Pr(\wt{\m O}^c)
 \ee
 where we have used that for $x\geq 3y/16 > 0$, 
 \be
 \frac{|x^p-y^p|}{x^p} = \frac{|x-y|}{x}\sum_{i=0}^{p-1}\l(\frac{y}{x}\r)^{p-i}\leq \l(\frac{16}{3}\r)^{p}\frac{|x-y|}{x}.
 \ee
 Moreover, 
 \ml{
  \mb E\left|\frac{\sigma^p}{\wh{\sigma}^p_n}\mathbf{1}\{\wt{\m O}\}-1\right|   \leq C_p \mb E\left|\frac{V_n^2 - \sigma^2}{V_n^2 + \sigma^2}\right|\mathbf{1}\{\wt{\m O}\} +c_p' \mb E\l(\frac{Q_n^2V^2_n}{V_n^2+\sigma^2}\mathbf{1}\{\wt{\m O}\}\r) + \Pr(\wt{\m O}^c) \\
    \leq C_p \mb E\left|\frac{V_n^2 - \sigma^2}{V_n^2 + \sigma^2}\right|\mathbf{1}\{\wt{\m O}\} + \frac{c'_{p}}{n} +  e^{-cn},
}
where we employed inequality \eqref{eq:elementary_2} and Lemma \ref{lem:sns-2}. 
Observing that the random variable $\left|\frac{V_n^2 - \sigma^2}{V_n^2 + \sigma^2}\right|\mathbf{1}\{\wt{\m O}\}$ converges to $0$  in probability (in view of the Law of Large Numbers) and is bounded, hence the convergence holds also in  $L^1$. 
This completes the proof.
\end{proof}


\subsection{Proof of Theorem \ref{thm:deviation}.}

Let $p\geq 1$. 
Denote $\wh{\mu}:=\wh{\mu}_{N,p}$ and consider the events 
\[
\mathcal{O}_j:=\{|Q_j| \leq 1/2\}\cap\{V_j \geq \sigma / 2\}.
\]
Set 
\ben\label{eq:event_E}
\m E: = \bigcap_{j=1}^k \mathcal{O}_j.
\een
Using Lemma \ref{lem:sns-2} and inequality \eqref{eq:sns-6}, we get that
\be
\Pr(\m E^c) \leq  ke^{-cn}
\ee
for some constant $c>0$ depending on the distribution of $X$. Therefore, for all $t > 0$
\[
\Pr(|\wh \mu_{N,p} - \mu| \geq  t) \leq \Pr(\{|\wh \mu_{N,p} - \mu| \geq  t\}\cap \m E \cap \Ep ) + ke^{-cn} + \Pr(\Ep^c).
\]
Recall the definition \eqref{eq:t-stat} of the t-statistics $T_1,\ldots,T_k$. The following chain of inequalities holds:
\ml{
\Pr(\{|\wh \mu_{N,p} - \mu| \geq  t\}\cap \m E \cap \Ep) \leq \Pr\l(\l\{ \l|\sum_{j=1}^k \frac{T_j}{\wh \sigma_j^{p-1}}   \r| \geq t \frac{k}{4^p\sigma^p} \r\} \cap \m E \r) \\
\leq \Pr\l( \l|\sum_{j=1}^k \frac{T_j}{\wh \sigma_j^{p-1}}\mathbf{1}\{\mathcal{O}_j\}   \r| \geq t \frac{k}{4^p\sigma^p} \r) \\
\leq \Pr\l( \l|\sum_{j=1}^k w_j   \r| \geq t \frac{k}{4^p\sigma} - k\l |\mb E\l( \frac{\sigma^{p-1}T_j}{\wh \sigma_j^{p-1}}\mathbf{1}\{\mathcal{O}_j\} \r) \r| \r) 
}
where $w_j := \frac{\sigma^{p-1}T_j}{\wh \sigma_j^{p-1}}\mathbf{1}\{\mathcal{O}_j\} - \mb E\l( \frac{\sigma^{p-1}T_j}{\wh \sigma_j^{p-1}}\mathbf{1}\{\mathcal{O}_j\} \r)$. 
It is easy to check that $\sqrt{n}w_j$ is a centered sub-Gaussian random variable, since in view of Lemma \ref{lem:concentration2} we have that for all $\lambda \in \mathbf{R}$,
\be
\mb E\l( e^{\sqrt{n}\lambda w_j}\r) \leq e^{c_p\lambda^2/2}
\ee
for some $c_p>0$ depending only on $p$. Choosing $t$ as
\be
t  = 4^{p}\sigma \l|\mb E\l( \frac{\sigma^{p-1}T_j}{\wh \sigma_j^{p-1}}\mathbf{1}\{\mathcal{O}_j\} \r) \r| + 4^p\sigma\sqrt{\frac{2c_ps}{N}},
\ee 
we get that 
\ml{
\Pr(\{|\wh \mu - \mu| \geq  t\}\cap \m E \cap \Ep) \leq   \Pr\l( \l|\sum_{j=1}^k w_j   \r| \geq k\sqrt{\frac{2c_ps}{N}} \r) \\
\leq 
\Pr\l( \sum_{j=1}^k \sqrt{\frac{2sN}{c_p}}w_j    \geq 2sk \r) + \Pr\l( -\sum_{j=1}^k \sqrt{\frac{2sN}{c_p}}w_j    \geq 2sk \r)\\
\leq 2 \l( \mb E \l[ e^{\sqrt{\frac{2s}{c_p k}}\sqrt{n}w_j}\r] \r)^k e^{-2s} \leq 2e^{-s},
}
where we used Chernoff bound on the last step. 
Combining the display above with Lemma \ref{lem:bias__variance_student}, we conclude that for all $s>0$
\be
\Pr\l(|\wh \mu - \mu| \geq  C_p\sigma \l(\phi(\delta,n) + \sqrt{\frac{s+ke^{-cn}}{N}}\r)\r) \leq  2e^{-s} + ke^{-cn} +\Pr(\Ep^c),
\ee
for some $C_p>0$ depending only on $p$. 
When $ke^{-cn}\geq 1$, the previous bound is trivial. It follows that 
\be
\Pr\l(|\wh \mu - \mu| \geq  C_p\sigma \l(\phi(\delta,n) + \sqrt{\frac{s+1}{N}}\r)\r) \leq  2e^{-s} + ke^{-cn}  +\Pr(\Ep^c)
\ee
for all $s>0$.

\subsection{Proof of Theorem \ref{thm:outlier}.}

The proof follows similar steps as the argument used to establish Theorem \ref{thm:deviation}. 
We will first show that with high probability the proportion of outliers in each bucket of observations is less than $1/2$. 
Indeed, letting $W_j$ denote the number of ouliers in the subsample indexed by $G_j$, it is straightfoward to see that 
$\sum_{j=1}^k W_j = O$, and that the random variables $\{W_j, \ j=1,\ldots,k\}$ are negatively correlated. 
Consider the event 
\[
\m E_2 = \bigcap_{j=1}^k\{ W_j \leq n/2 \}.
\]
 Recall that $W_j = \sum_{i \in G_j} \mathbf{1}_{i \in \mathcal{O}_j}$. Since $\sum_{j=1}^k W_j = O$,  the random variables $(\mathbf{1}_{i \in \mathcal{O}_j})_{i \in G_j}$ are $1$-negatively correlated for each $j=1,\dots,k$, as a sub-sequence of a $1$-negatively correlated sequence of random variables. Applying the Chernoff bound for negatively correlated random variables (see \citep[][section 1.10.2.2 and Theorem 1.10.23]{doerr2020probabilistic} for the definitions and the required version of the Chernoff bound), we get that as long as $ O \leq N/4$, 
\be
\Pr( \m E_2^c) \leq ke^{-cn}.
\ee
Hence in what follows, we can restrict our attention on the event $\m E_2$. 
We use the superscript $I$ to denote ``clean'' sample and $C$ (``corrupted'') -- otherwise. Notice that
\be
\bar{\mu}_j - \mu = \frac{W_j}{n}(\bar{\mu}^C_j - \mu)+\l(1 - \frac{W_j}{n}\r)(\bar{\mu}^{I}_j - \mu) = \frac{W_j}{n}(\bar{\mu}^C_j - \bar{\mu}_j^{I})+ \bar{\mu}^{I}_j - \mu
\ee
where $\bar{\mu}_j^C, \bar{\mu}_j^I$ are, respectively, empirical means of the corrupted and clean part of the sub-sample indexed by $G_j$. We also have that
\ben\label{eq:variance_outlier}
\wh{\sigma}_j^2 = \frac{W_j}{n} (\wh{\sigma}_j^C)^2 + \l(1-\frac{W_j}{n}\r) (\wh{\sigma}_j^I)^2 + \frac{W_j(n-W_j)}{n^2}(\bar{\mu}_j^C - \bar{\mu}_j^I)^2,
\een
where $(\wh{\sigma}_j^C)^2, (\wh{\sigma}_j^I)^2$ are, respectively, empirical variances of the corrupted and clean sub-samples of $G_j$. 
Observe that $\wh{\sigma}_j^2 \geq (\wh{\sigma}_j^I)^2/2$, and, therefore, as in the previous proof we deduce that the weights $\alpha_j$ given by $\eqref{eq:weights}$ can not be too large even when outliers are present in the sample. 
Consider the events $\mathcal{O}_j := \{|Q^{I}_j| \leq 1/2\}\cap\{V_j^I \geq \sigma / 2\}\cap \{ W_j \leq n/2 \} $, and 
\be
\m E := \bigcap_{j=1}^k \mathcal{O}_j .
\ee
Using Lemma \ref{lem:sns-2} and inequality \eqref{eq:sns-6}, we get that
\be
\Pr(\m E ^c) \leq ke^{-cn}
\ee
for some constant $c>0$ that depends only on the distribution of $X$. 
In the rest of the proof we assume that the event $\m E\cap \m E_p$ holds, with $\Ep$ defined in \eqref{eq:Ep}. 
On this event, we have that 
\ml{ 
\l| \wh{\mu} - \mu \r| \mathbf{1}\{\m E\} \leq \l(\frac{4\sigma}{1-C}\r)^p\underbrace{\l| \frac{1}{k}\sum_{j = 1}^k\frac{W_j(\bar{\mu}^{C}_j - \bar{\mu}_j^I) }{n\wh{\sigma}^p_j}\mathbf{1}\{\mathcal{O}_j\}\r|}_{(A)} \\
+ \l(\frac{4\sigma}{1-C}\r)^p\underbrace{\l|\frac{1}{k}\sum_{j = 1}^k\frac{\bar{\mu}^I_j - \mu}{\wh{\sigma}^p_j}\mathbf{1}\{\mathcal{O}_j\} - \mb E\l( \frac{\bar{\mu}^I_j - \mu}{\wh{\sigma}^p_j}\mathbf{1}\{\mathcal{O}_j\}\r)\r|}_{(B)} \\
+ \l(\frac{4\sigma}{1-C}\r)^p \underbrace{\frac{1}{k}\sum_{j=1}^k\l|\mb E\l( \frac{\bar{\mu}^I_j - \mu}{\wh{\sigma}^p_j}\mathbf{1}\{\mathcal{O}_j\}\r)\r|}_{(C)}.
}
We will proceed by estimating each of the terms separately.
\begin{description}
    \item[Control of (A):] 
    Using \eqref{eq:variance_outlier}, observe that on $\m O_j$ we have
    \be 
    \wh{\sigma}^2_j \geq \frac{(\wh\sigma_j^I)^2}{2} + \frac{W_j}{2n}(\bar{\mu}^C_j - \bar{\mu}^I_j)^2 \geq C'\l(\sigma^2 + \frac{W_j}{n}(\bar{\mu}^C_j - \bar{\mu}^I_j)^2 \r),
    \ee
   for some absolute constant $C'>0$. It comes out that
    \ml{
    \left| \frac{1}{k}\sum_{j = 1}^k\frac{W_j(\bar{\mu}^{C}_j - \bar{\mu}_j^I) }{n\wh{\sigma}^p_j}\mathbf{1}\{\mathcal{O}_j\} \right| \leq \frac{1}{k}\sum_{j=1}^k\frac{C_pW_j|\bar{\mu}^C_j - \bar{\mu}^I_j| }{ n\sqrt{ \sigma^2 + \frac{W_j}{n}(\bar{\mu}_j^C - \bar{\mu}_j^I)^2}^p} \\
    \leq \frac{C_p}{k\sigma^{p-1}}\sum_{j=1}^k\frac{W_j|\bar{\mu}^C_j - \bar{\mu}^I_j| }{ n\sqrt{ \sigma^2 + \frac{W_j}{n}(\bar{\mu}_j^C - \bar{\mu}_j^I)^2}} \\
    \leq \frac{C_p\alpha(\m O)^{(1-p)/2}}{k\sigma^{p-1}}\sum_{j=1 }^k\sqrt{W_j/n},
    }
    where $\alpha(\m O):= 1+ \underset{j/ W_j \neq 0}{\min}\frac{W_j(\bar{\mu}_j^C - \bar{\mu}_j^I)^2}{n\sigma^2}$. 
    Hence it follows from Cauchy-Schwarz inequality that
    \be
     \left| \frac{1}{k}\sum_{j = 1}^k\frac{W_j(\bar{\mu}^{C}_j - \bar{\mu}_j^I) }{n\wh{\sigma}^p_j}\mathbf{1}\{\mathcal{O}_j\} \right| 
    \leq\frac{2^p\alpha(\m O)^{(1-p)/2}}{\sigma^{p-1}}\frac{\sqrt{\sum_{j =1 }^kW_j}(\sqrt{O}\wedge \sqrt{k})}{k\sqrt{n}}.
    \ee
   As a consequence,
    \be
   \left| \frac{1}{k}\sum_{j = 1}^k\frac{W_j(\bar{\mu}^{C}_j - \bar{\mu}_j^I) }{n\wh{\sigma}^p_j}\mathbf{1}\{\mathcal{O}_j\} \right| \leq \frac{2^p\alpha(\m O)^{(1-p)/2}}{\sigma^{p-1}}\l(\frac{O}{k\sqrt{n}} \wedge \sqrt{\frac{O}{N}}\r).
    \ee
    Observe that the previous statement holds pointwise, and is not probabilistic in nature. 
    It also suggests that the worst scenario occurs whenever all buckets are corrupted.
    \item[Control of (B):]
    Since $\wh{\sigma}_j^2 \geq \l(\wh{\sigma}_j^I\r)^2/2$ under $\m O_j$, we have that
    \be 
    \Pr\l( \sigma^{p-1}\l|\frac{\bar{\mu}^I_i - \mu}{\wh{\sigma}^p_i}\r|\mathbf{1}\{\mathcal{O}_j\} \geq x \r) \leq \Pr\l( \sigma^{p-1}\l|\frac{\bar{\mu}^I_i - \mu}{\l(\wh{\sigma}_j^I\r)^p}\r|\mathbf{1}\{\mathcal{O}_j\} \geq 2^{p/2}x \r).
    \ee
    Hence we can show, as in Lemma \ref{lem:concentration2}, that the random variable $$\sigma^{p-1}\frac{\bar{\mu}^I_i - \mu}{\wh{\sigma}^p_i}\mathbf{1}\{\mathcal{O}_j\}$$ is sub-Gaussian. 
    Following the same arguments as in Theorem \ref{thm:deviation}, this leads to the bound
    \be
    \sigma^{p-1}\left|\frac{1}{k}\sum_{j = 1}^k\frac{\bar{\mu}^I_i - \mu}{\wh{\sigma}^p_i}\mathbf{1}\{\mathcal{O}_j\} - \mb E\l( \frac{\bar{\mu}^I_i - \mu}{\wh{\sigma}^p_i}\mathbf{1}\{\mathcal{O}_j\}\r)\right| \leq C_{p}\sqrt{\frac{s}{N}}
    \ee
   that holds with probability at least $1-2e^{-s}$ for some $C_p>0$.
    \item[Control of (C):]
    As for the ``bias term,'' it is enough to observe that for uncorrupted buckets, $\wh{\sigma}_j = \wh{\sigma}^{I}_j$, and the bias can be upper bounded exactly as in Theorem \ref{thm:deviation}. Hence
    \be
    \frac{\sigma^{p-1}}{k}\sum_{j\in I}\l|\mb E\l( \frac{\bar{\mu}^I_j - \mu}{\wh{\sigma}^p_j}\mathbf{1}\{\mathcal{O}_j\}\r)\r| \leq \phi(\delta,n) + C_p\sqrt{\frac{ke^{-cn}}{N}}.
    \ee
    At the same time, for the corrupted part of the bias term, we have on $\m O_j$ that 
    \be 
    \wh{\sigma}^p_j \geq C'_p\wh{\sigma}_j\l(\sigma^2 + \frac{W_j}{n}(\bar{\mu}^C_j - \bar{\mu}^I_j)^2\r)^{(p-1)/2},
    \ee
    for some $C_p'>0$ depending only on $p$. Hence
    \ml{
       \frac{\sigma^{p-1}}{k} \l|\sum_{j \in C}\mb E\l( \frac{\bar{\mu}^I_j - \mu}{\wh{\sigma}^p_j}\mathbf{1}\{\mathcal{O}_j\}\r)\r|
       \\
       \leq \frac{C_p\sigma^{p-1}}{k}\sum_{j \in C}\mb E\l| \frac{\bar{\mu}^I_j - \mu}{\wh{\sigma}^I_j (\sigma^2 + W_j/n(\bar{\mu}_j^C - \bar{\mu}_j^I)^2)^{(p-1)/2}}\mathbf{1}\{\mathcal{O}_j\}\r| \\
        \leq \frac{C_p\alpha(\m O)^{(1-p)/2}(O \wedge k)}{k}\mb E|T^I_{1}|\mathbf{1}\{ \m O_1\}
        \leq \frac{C_p\alpha(\m O)^{(1-p)/2}(O \wedge k)}{k\sqrt{n}} 
        \\
        \leq C_p\alpha(\m O)^{(1-p)/2}\l(\frac{O}{k\sqrt{n}} \wedge \sqrt{\frac{O}{N}}\r),
    }
    for some $C_p>0$, where we have used inequality \eqref{eq:elementary_2} and the fact that $O \wedge k \leq O \wedge \sqrt{Ok}$.
\end{description}
This concludes the proof of the fact that with probability at least $1-2e^{-s} - ke^{-cn} - \Pr(\Ep^c)$,
\ben
\label{eq:robust:neat}
|\wh\mu_{N,p} - \mu| \leq \frac{C_p\sigma}{(1-C)^p} \l( \sqrt{\frac{s+1}{N}} +   \phi(\delta,n)+\alpha(\m O)^{(1-p)/2} \l(\frac{\m O}{k\sqrt{n}} \wedge \sqrt{\frac{O}{N}}\r)\r).
\een

\subsection{Proof of Theorem \ref{thm:clt}.}

Using the definition of $\phi$, it is easy to see that $\sqrt{N_j}\phi(\delta,n_j) = o(1)$, implying that $k_j = o(n_j)$ which in turn implies that $k_j = o\l( e^{cn_j}\r)$ for any constant $c>0$.
We recall that
\be
\wh \mu_{N_j,p} - \mu =  \frac{\frac{1}{k_j}\sum_{i=1}^{k_j} \frac{T_i}{\wh \sigma_i^{p-1}} }{\frac{1}{k_j}\sum_{i=1}^{k_j} \frac{1}{\wh \sigma_i^p} }.
\ee
Next, we will use the following decomposition that holds on the event $\mathcal{E}$ \eqref{eq:event_E} defined in the proof of Theorem \ref{thm:deviation}.
\ml{
\sqrt{N_j}(\wh \mu - \mu)\mathbf{1}\{\mathcal{E}\} 
\\
= H\frac{\sqrt{n_j\var\l(\frac{T_1}{\wh \sigma_1^{p-1}}\mathbf{1}\{\mathcal{O}_1\}\r)}
\frac{\sum_{i=1}^{k_j} \l(\frac{T_i}{\wh \sigma_i^{p-1}}\mathbf{1}\{\mathcal{O}_i\} - \mb E\frac{T_i}{\wh \sigma_i^{p-1}}\mathbf{1}\{\mathcal{O}_i\}\r)}{\sqrt{\sum_{i=1}^{k_j}\var\l(\frac{T_i}{\wh \sigma_i^{p-1}}\mathbf{1}\{\mathcal{O}_i\}\r)}} + \sqrt{N_j}\mb E\frac{T_1}{\wh \sigma_1^{p-1}}\mathbf{1}\{\mathcal{O}_1\} }
{\frac{1}{k_j}\sum_{i=1}^{k_j} \l( \frac{1}{\wh \sigma_i^p}\mathbf{1}\{\mathcal{O}_i\}- \mb E\frac{1}{\wh \sigma_i^p}\mathbf{1}\{\mathcal{O}_i\} \r) + \mb E \l(\frac{1}{\wh \sigma_1^p}\mathbf{1}\{\mathcal{O}_1\} \r)}
}
where $H=\mathbf{1}\{\mathcal{E}\}. $Using Lemma \ref{lem:asymptotic-variance}, we have that
\be
\mb E\frac{1}{\wh \sigma_1^p}\mathbf{1}\{\mathcal{O}_1\} \xrightarrow{j \to \infty} \sigma^{-p}.
\ee
Moreover using Lemma \ref{lem:bias__variance_student} we have
\be
\sqrt{N_j}\l|\mb E\frac{T_1}{\wh \sigma_1^{p-1}}\mathbf{1}\{\mathcal{O}_1\} \r| \leq  \sqrt{N_j}\phi(\delta,n_j) + \sqrt{k_je^{-cn_j}} \xrightarrow{j \to \infty} 0,
\ee
and 
\be
\sqrt{n_j\var\l(\frac{T_1}{\wh \sigma_1^{p-1}}\mathbf{1}\{\mathcal{O}_1\}\r)} \xrightarrow{j \to \infty} \sigma^{1-p}.
\ee
Since the independent variables $\frac{T_i}{\wh{\sigma}_{i}^{p-1}}\mathbf{1}\{\mathcal{O}_i\}$ and $\frac{1}{\wh{\sigma}_{i}^{p}}\mathbf{1}\{\mathcal{O}_i\}$ are uniformly bounded, they satisfy Lindeberg's condition. Therefore, 
\be
\frac{\sum_{i=1}^{k_j} \l(\frac{T_i}{\wh \sigma_i^{p-1}}\mathbf{1}\{\mathcal{O}_i\} - \mb E\frac{T_i}{\wh \sigma_i^{p-1}}\mathbf{1}\{\mathcal{O}_i\}\r)}{\sqrt{\sum_{i=1}^{k_j}\var\l(\frac{T_i}{\wh \sigma_i^{p-1}}\mathbf{1}\{\mathcal{O}_i\} \r)}}  \xrightarrow{j \to \infty}\mathcal{N}(0,1)
\ee
in distribution, and 
\be
\frac{1}{k_j}\sum_{i=1}^{k_j} \l(\frac{1}{\wh \sigma_i^p}\mathbf{1}\{\mathcal{O}_i\}- \mb E \l(\frac{1}{\wh \sigma_i^p}\mathbf{1}\{\mathcal{O}_i\} \r)\r) \xrightarrow{j \to \infty} 0
\ee
in probability. In addition, we have that
\be
\Pr(\mathcal{E}^c) \leq k_je^{-n_j} \xrightarrow{j \to \infty} 0,
\ee
established as in the proof of Theorem \ref{thm:deviation}. Putting everything together, we finally conclude that 
\be
\sqrt{N_j}\l(\wh \mu_{N_j,p} - \mu\r)\xrightarrow{d} \mathcal{N}(0,\sigma^{2})
\ee
in distribution as $j \to \infty$. 

\subsection{Proof of Theorem \ref{thm:outlier_adaptive}.}
For any integer $m$, we denote by $\Ep(m)$ the event $\Ep$ defined via \eqref{eq:Ep} with $m$ blocks. For every event $\m A$, $\m A^c$ will denote its complementary.  
Observe that, as long as $1/20 \leq \frac{\wt{\sigma}}{\sigma} \leq 4$, we have the following inclusions
\be
\{ (\wt{k}\vee s) \geq 3 (O \vee s) \} \subset \{ \wt{k} \geq 3 (O \vee s) \} \subset \{ \wt{\Ep}(\lfloor 3 (O \vee s)/2\rfloor)^{c}\}   \subset \{ \Ep(\lfloor 3 (O \vee s)/2\rfloor)^{c}\} 
\ee
where $\lfloor 3 (O \vee s)/2\rfloor$ denotes the integer part of $3  (O \vee O)/2$.
Therefore, we deduce, using Lemma \ref{lem:denominator}, that 
\be
\Pr\l(   (\wt{k} \vee s) \geq 3 (O \vee s) \r) \leq e^{-c_2 (O \vee s)}.
\ee
Finally, we recall that when $\Ep$ holds and $k \leq 3 (O\vee s)$, then with probability at least $1- 2e^{-s} - k e^{-cN/ (O\vee s)}$  
\be
|\wh\mu_{N,p} - \mu| \leq C_p\sigma \l( \sqrt{\frac{s}{N}} + \sqrt{\frac{O}{N}}\r),
\ee
as shown in \eqref{eq:robust:neat}. Combining the previous results, we conclude that 
\ml{
\Pr\l( |\wt{\mu}_p(s) - \mu| \geq C_p\sigma \l( \sqrt{\frac{s}{N}} + \sqrt{\frac{ O}{N}}\r) \r) \\
\leq \Pr\l( \l\{|\wt{\mu}_p(s) - \mu| \geq C_p\sigma \l( \sqrt{\frac{s}{N}} + \sqrt{\frac{ O}{N}}\r) \r\} \cap \{ \wt{k} \leq 3 (O\vee s) \}  \r) + e^{-c_2 (O \vee s)}  \\
\leq \sum_{i=1}^{\log_2{(3 (O \vee s))}}(2e^{-s} + 2^{i}e^{-cN/ (O\vee s)}) + e^{-c_2 (O\vee s)} 
\\ \leq 2\log_2{(3 O)}e^{-s} + e^{-c  (O\vee s)} +  (O\vee s)e^{-cN/ (O \vee s)},
}
where we used in the last inequality the fact that $\{ (\wt{k} \vee s) \leq 3 (O \vee s) \}  \subset \l\{\l(\frac{1}{k}\sum_{j=1}^k \frac{1}{\wh \sigma^p_j}\r)^{-1}  \leq \l(\frac{160\sigma}{1-C}\r)^{p}\r\}$, so that event $\Ep$ holds.

\subsection{Construction of a robust estimator of $\sigma$.}
\label{app:robust_variance}

Let $N \geq 400$. Without loss of generality, we can assume that $N = 100 k$ where $k$ is an integer and that $\{1,\dots,N\} = \wt{G}_1 \cup \dots \cup \wt{G}_{k} $ where $\wt{G}_j = \{ 100(j-1)+1, \dots , 100j\}$ for all $j=1,\dots,k$. 
Let $\wt{\sigma}$ be defined as follows:
\be
\wt{\sigma} := \med\l(\frac{1}{50}\sum_{2i \in \wt{G}_1}|X_{2i} - X_{2i-1}|, \dots ,\frac{1}{50}\sum_{2i \in \wt{G}_k}|X_{2i} - X_{2i-1}|\r)
\ee
Under two moments assumption, the following result holds.
\begin{lemma}
\label{lemma:var}
Assume that $\mb E|X-\mu|^2 < \infty$, $\mb E|X-\mu| \geq \sigma /2$ and that $ O \leq N/400$. Then, with probability at least $1- e^{-cN}$, we have that
\be 
1/20 \leq \frac{\wt{\sigma}}{\sigma} \leq 4,
\ee
where $c>0$ is an absolute constant.
\end{lemma}
Note that Lemma \ref{lemma:var} requires the new condition $\mb E|X-\mu| \geq \sigma /2$. The latter condition is mild and can be viewed as the equivalence between absolute first and second moments which is less restrictive than the equivalence between centered moments of order $2$ and $2+\delta$. This condition may also be seen as the price to pay for adaptation under only two moments.
\begin{proof}
Using Jensen's inequality, we get that $\mb E|X_1 - X_2| \geq \mb E|X - \mu| \geq \sigma /2$ and $\mb E |X_1 - X_2| \leq 2 \sigma$. Therefore,
\be 
\Pr\l(  1/20 \leq \frac{\wt{\sigma}}{\sigma} \leq 4 \r) \geq \Pr\l( \l| \frac{\wt{\sigma}}{\mb E|X_1 - X_2|} - 1 \r| \leq 9/10 \r).
\ee
Since 
\ml{
\Pr\l( \l|\frac{1}{50 \mb E|X_1 - X_2|}\sum_{2i \in \wt{G}_1}|X_{2i} - X_{2i-1}| - 1 \r| \geq 9/10 \r) 
\\ \leq \frac{200 \sigma^2}{4050 (\mb E|X_1 - X_2|)^2} \leq 1/5
}
and that $ O \leq k/4$, we conclude that
\ml{
\Pr\l( \l| \frac{\wt{\sigma}}{\mb E|X_1 - X_2|} - 1 \r| \geq 9/10 \r) \leq \Pr\l( \sum_{j=1}^{k- O} Z_{j}\geq  k/4 \r) \\
\leq \Pr\l( \sum_{j=1}^{k- O} (Z_{j} - \mb E Z_{j}) \geq  k/20 \r) \leq e^{-cN},
}
where $Z_{j}:= \mathbf{1}\l\{ \l|\frac{1}{50 \mb E|X_1 - X_2|}\sum_{2i \in \wt{G}_1}|X_{2i} - X_{2i-1}| - 1 \r| \geq 9/10\r\}$ and $c>0$.
\end{proof}

\section*{Acknowledgements}

Authors acknowledge support by the National Science Foundation grants DMS-1712956 and CIF-1908905. M.N. was partially supported by a James H. Zumberge Faculty Research and Innovation Fund at the University of Southern California.



\bibliographystyle{alpha}
\newcommand{\etalchar}[1]{$^{#1}$}


\end{document}